\documentclass[11pt,letterpaper]{article}
\usepackage{graphicx}
\usepackage{epsfig,subfigure}
\usepackage{amssymb}
\usepackage{amsthm}
\usepackage{amsfonts}
\usepackage{amsmath}
\usepackage{multicol,multirow}
\usepackage[margin=1in]{geometry}

\newtheorem{theorem}{Theorem}
\newtheorem{lemma}[theorem]{Lemma}

\newtheorem{proposition}[theorem]{Proposition}
\newtheorem{definition}{Definition}

\newtheorem{problem}{Problem}

\newcommand{\el}{\mathrm{expaff}}
\newcommand{\ones}{\mathbf{1}}
\newcommand{\bQ}{\mathbf{Q}}
\newcommand{\cQ}{\mathcal{Q}}
\newcommand{\bR}{\mathbf{R}}
\newcommand{\bS}{\mathbf{S}}

\newcommand{\bA}{\mathbf{A}}
\newcommand{\bC}{\mathbf{C}}
\newcommand{\bZ}{\mathbf{Z}}
\newcommand{\cC}{\mathcal{C}}
\newcommand{\cL}{\mathcal{L}}
\newcommand{\cD}{\mathcal{D}}

\newcommand{\cF}{\mathcal{F}}
\newcommand{\bx}{\mathbf{x}}
\newcommand{\by}{\mathbf{y}}
\newcommand{\bz}{\mathbf{z}}
\newcommand{\bc}{\mathbf{c}}

\newcommand{\bd}{\mathbf{d}}

\newcommand{\bq}{\mathbf{q}}
\newcommand{\br}{\mathbf{r}}
\newcommand{\bs}{\mathbf{s}}
\newcommand{\bp}{\mathbf{p}}
\newcommand{\bu}{\mathbf{u}}
\newcommand{\bv}{\mathbf{v}}
\newcommand{\bw}{\mathbf{w}}
\newcommand{\bg}{\mathbf{g}}
\newcommand{\bF}{\mathbf{F}}
\newcommand{\bh}{\mathbf{h}}
\newcommand{\bM}{\mathbf{M}}
\newcommand{\bN}{\mathbf{N}}
\newcommand{\bzero}{\mathbf{0}}
\newcommand{\R}{\mathbb{R}}
\newcommand{\Sp}{\mathbb{S}}
\newcommand{\Sym}{\mathrm{Sym}}
\newcommand{\K}{\mathcal{K}}
\newcommand{\bnu}{\boldsymbol \nu}
\newcommand{\blambda}{\boldsymbol \lambda}
\newcommand{\bdelta}{\boldsymbol \delta}
\newcommand{\bmu}{\boldsymbol \mu}
\newcommand{\bzeta}{\boldsymbol \zeta}
\newcommand{\bbeta}{\boldsymbol \beta}
\newcommand{\btheta}{\boldsymbol \theta}
\newcommand{\xinit}{\mathcal{X}_\mathrm{initial}}
\newcommand{\xtarget}{\mathcal{X}_\mathrm{target}}

\title{Conic Geometric Programming}
\author{Venkat Chandrasekaran$^c$ and Parikshit Shah$^w$ \thanks{Email: venkatc@caltech.edu, pshah@discovery.wisc.edu} \vspace{0.25in} \\ $^c$ Departments of Computing and Mathematical Sciences and of Electrical Engineering\\ California Institute of Technology \\ Pasadena, CA 91125  \vspace{0.1in} \\ $^w$ Wisconsin Institutes for Discovery\\ University of Wisconsin\\ Madison, WI 53715}
\date{October 10, 2013}

\begin{document}

\maketitle

\begin{abstract}
We introduce and study \emph{conic geometric programs} (CGPs), which are convex optimization problems that unify geometric programs (GPs) and conic optimization problems such as semidefinite programs (SDPs).  A CGP consists of a linear objective function that is to be minimized subject to affine constraints, convex conic constraints, and upper bound constraints on sums of exponential and affine functions.  The conic constraints are the central feature of conic programs such as SDPs, while upper bounds on combined exponential/affine functions are generalizations of the types of constraints found in GPs.  The dual of a CGP involves the maximization of the negative relative entropy between two nonnegative vectors jointly, subject to affine and conic constraints on the two vectors.  Although CGPs contain GPs and SDPs as special instances, computing global optima of CGPs is not much harder than solving GPs and SDPs.  More broadly, the CGP framework facilitates a range of new applications that fall outside the scope of SDPs and GPs.  Specifically, we demonstrate the utility of CGPs in providing solutions to problems such as permanent maximization, hitting-time estimation in dynamical systems, the computation of the capacity of channels transmitting quantum information, and robust optimization formulations of GPs.
\end{abstract}


\section{Introduction} \label{sec:intro}
Geometric programs (GPs) \cite{BoyKVH2007,DufPZ1967} and semidefinite programs (SDPs) \cite{VanB1994} are prominent classes of structured convex optimization problems that each generalize linear programs (LPs) in different ways.  By virtue of convexity, GPs and SDPs possess favorable analytical and computational properties: a rich duality theory and polynomial-time algorithms for computing globally optimal solutions.  Further, due to their flexible modeling capabilities, GPs and SDPs are useful in a range of problems throughout the sciences and engineering.  Some prominent applications of SDPs include relaxations for combinatorial optimization \cite{Ali1995,GoeW1995,KarMS1998,Vaz2004}, ellipsoid volume optimization via determinant maximization \cite{VanBW1998}, statistical estimation \cite{Sha1982}, problems in control and systems theory \cite{BoyEFB1994}, and matrix norm optimization \cite{RecFP2010,VanB1994}.  On the other hand, applications of GPs include the computation of information-theoretic quantities \cite{ChiB2004}, digital circuit gate sizing \cite{BoyKPH2005}, chemical process control \cite{WalGW1986}, matrix scaling and approximate permanent computation \cite{LinSW2000}, entropy maximization problems in statistical learning \cite{DinKW1977}, and power control in communication systems \cite{Chi2005}.  In this paper we describe a new class of tractable convex optimization problems called \emph{conic geometric programs} (CGPs), which conceptually unify GPs and conic programs such as SDPs, thus facilitating a broad range of new applications.

\paragraph{Motivation} Although SDPs and GPs provide effective solutions for many applications, natural generalizations of these applications that arise in practice fall outside the scope of GPs and SDPs.

\begin{enumerate}
\item{\bf Permanent maximization} Given a collection of matrices, find the one with the largest permanent. Computing the permanent of a matrix is a well-studied problem that is thought to be computationally intractable.  Therefore, we will only seek approximate solutions to the problem of permanent maximization.

\item{\bf Robust optimization} The solution of a GP is sensitive to the input parameters of the problem.  Compute GPs within a robust optimization framework so that the solutions offer robustness to variability in the problem parameters.

\item{\bf Hitting times in dynamical systems} Given a linear dynamical system consisting of multiple modes and a region of feasible starting points, compute the smallest time required to hit a specified target set from an arbitrary feasible starting point.

\item{\bf Quantum channel capacity} Compute the capacity of a noisy communication channel over which quantum states can be transmitted.  The capacity of a channel is the maximum rate at which information can be communicated per use of the channel \cite{Hol1997,SchW1998}.
\end{enumerate}

A detailed mathematical description of each of these applications is given in Section~\ref{sec:applications}.  The pursuit of a unified and tractable solution framework for these questions motivates our development of CGPs in this paper.  Some of these applications and close variants have been studied previously in the literature; we discuss the relevant prior work and the more powerful generalizations afforded by CGPs in Section~\ref{sec:applications}.  We also describe in Section~\ref{sec:applications} other problems in domains such as quantum state tomography, statistical mechanics (computation of equilibrium densities), and kernel learning in data analysis for which previously proposed methods can be viewed as special instances of CGPs.

A GP is a convex optimization problem in which a linear function of a decision variable $\bx \in \R^n$ is minimized subject to affine constraints on $\bx$ and upper-bound constraints on \emph{positive sums of exponentials} of affine functions of $\bx$.  Formally, these latter constraints are specified as follows\footnote{GPs are sometimes described in the literature as optimization problems involving so-called \emph{posynomial} functions, which are in general non-convex; details of the equivalence between the posynomial description of GPs and their specification in convex form can be found in \cite{BoyKVH2007,DufPZ1967}.}:
\begin{equation*}
\sum_{i=1}^k \bc_i \exp([\bQ' \bx]_i) + d \leq 0.
\end{equation*}
Here $\bQ \in \R^{n \times k}$, $\bc \in \R^k$ is a vector with nonnegative entries ($\bc_i$ is the $i$'th entry of $\bc$), and $d \in \R$; all these quantities are fixed problem parameters.  Note that if there are no constraints on sums-of-exponentials, then a GP reduces to an LP.  Therefore, GPs include LPs as a special case.  An SDP is a prominent instance of a \emph{conic} optimization problem that generalizes LPs in a different manner in comparison to GPs.  In a conic program, the goal is to minimize a linear function of a decision variable subject to affine constraints as well as constraints that require the decision variable to belong to a convex cone.  Examples of conic optimization problems include LPs, second-order cone programs or SOCPs, and SDPs. We note here that GPs can also be viewed as conic programs for an appropriately chosen cone \cite{Gli2000}, although they are not usually described in this fashion. In an SDP the decision variable is a symmetric matrix and the conic constraint requires the matrix to be positive semidefinite.  If the matrix is further constrained to be diagonal, then an SDP reduces to an LP.  Thus, both SDPs and GPs are generalizations of LPs, as demonstrated in Figure~\ref{fig:lpsdpcgp}.  However, neither SDPs nor GPs contain the other as a special case, and we expand upon this point next.


\begin{figure}
\centering
\includegraphics[scale=0.35]{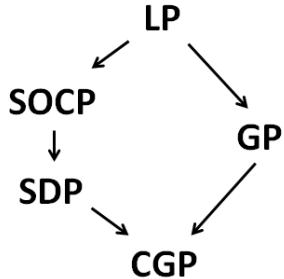} \caption{A graphical description of the various families of structured convex optimization problems ordered by inclusion (the arrows denote inclusion).}
\label{fig:lpsdpcgp}
\end{figure}

While GPs and SDPs have been employed successfully across many application domains, they also have certain restrictions that limit their applicability in solving our motivating problems.  For example, GPs do not in general allow for matrix decision variables on which spectral constraints can be imposed, e.g., a positive semidefiniteness condition on a symmetric matrix.  On the other hand, SDPs face the basic restriction that the constraint sets must have algebraic boundaries, i.e., these are specified by the vanishing of systems of polynomial equations.  Consequently, transcendental quantities such as the entropy function -- which are naturally expressed via GPs -- are inexpressible using SDPs.  These restrictions of GPs and SDPs limit their application in domains such as quantum information in which decision variables are usually specified as positive semidefinite matrices, but the objective functions are variants of the entropy function.




\paragraph{Conic geometric programs} These limitations inform our development in Section~\ref{sec:defn} of CGPs, which are obtained by blending the salient features of GPs and of conic programs such as SDPs.  A CGP is an optimization problem in which a linear function of a decision variable $\bx \in \R^n$ is minimized subject to affine constraints and two additional types of constraints.  First, we allow conic constraints that require $\bx$ to belong to a convex cone.  This addresses the restriction with GPs as the conic constraint can, for example, require a matrix decision variable to be positive semidefinite.  In order to address the second restriction, namely the inexpressibility of transcendental quantities using SDPs, we consider convex constraints in $\bx$ that combine an exponential function and an affine function as follows:
\begin{equation}
\exp(\bq' \bx  - s) + \br' \bx - t \leq 0. \label{eq:expaff}
\end{equation}
Here $\bq, \br \in \R^n$ and $s,t \in \R$ are fixed problem parameters.  Note that this constraint is more general than the pure sum-of-exponentials constraint allowed in GPs, as \eqref{eq:expaff} can be used to specify upper-bound constraints on combined sums of exponential and affine functions.  Duffin et al \cite{DufPZ1967} mention this generalization of GPs in their initial work (see also \cite{BoyKVH2007}), but such constraints are usually not considered as part of the standard form of a GP.  While constraints of the type \eqref{eq:expaff} may appear to be a modest extension of GPs, we will show in Section~\ref{sec:applications} that such combined exponential/affine constraints along with conic constraints significantly expand the expressive power of CGPs relative to GPs and SDPs, and consequently broaden the range of applications.  In particular, constraints of the form \eqref{eq:expaff} enable effective \emph{lift-and-project} representations in which complicated convex sets in $\R^n$ are specified as the projection of efficiently described sets in $\R^{n'}$ with $n' > n$ (see Section~\ref{sec:cgppower}).  Such lift-and-project techniques employed in the CGP framework are very useful in obtaining solutions to GPs that are robust to parameter uncertainty and in approximately maximizing the permanent over a family of matrices.  In summary, as described in Figure~\ref{fig:gp-sdp-cgp}, a CGP is an optimization problem in which a linear function is minimized subject to three types of constraints: combined exponential/affine constraints \eqref{eq:expaff}, affine constraints, and convex conic constraints.  Thus, CGPs are convex optimization problems, and they contain GPs, SOCPs, and SDPs as special cases as shown in Figure~\ref{fig:lpsdpcgp}.


The convexity of CGPs leads to favorable analytical properties, and these are especially apparent via convex duality.  As described in Section~\ref{sec:defn}, the dual of a CGP has a particularly appealing form: it is a convex optimization problem in which the negative of the \emph{relative entropy} between two nonnegative vectors is jointly maximized subject to affine and conic constraints on the vectors.  Recall that the relative entropy between two nonnegative vectors is a jointly convex function.  The conic constraint is expressed in terms of the dual of the cone that specifies the constraint in the primal CGP.  The additional expressive capability of CGPs in contrast to GPs and SDPs is also apparent via the respective dual formulations.  Figure~\ref{fig:gp-sdp-cgp} provides a concise summary -- see Section~\ref{sec:defn} for more details.


In addition to containing GPs and SDPs as special instances, CGPs offer significantly more expressive power than either of these families of convex programs considered separately.  For example, GPs can be used to solve entropy maximization problems involving probability vectors, while CGPs can be used to solve entropy maximization problems involving quantum density matrices.  Another example is a constraint in a decision variable $\bx \in \R^n$ of the form:
\begin{equation*}
\sum_{i=1}^k \bC^{(i)} \exp([\bQ' \bx]_i) \preceq \mathbf{I},
\end{equation*}
where the $\bC^{(i)}$ are positive semidefinite matrices, $\mathbf{I}$ is an identity matrix of appropriate size, $\bQ \in \R^{n \times k}$, and $\preceq$ denotes the matrix semidefinite ordering.  Such constraints allow us to minimize the largest eigenvalue of matrix expressions of the form $\sum_{i=1}^k \bC^{(i)} \exp([\bQ' \bx]_i)$ subject to additional convex constraints on $\bx$.  These types of expressions cannot be handled directly via GPs or SDPs alone; see Section~\ref{sec:cgppower} for more examples of convex constraints that can be expressed via CGPs.

\begin{figure*}
\centering
\small
\begin{tabular}{|c|c|c|c|c|}
\hline
& & {\bf Conic} & {\multirow{2}{*}{\bf GPs}} & {\multirow{2}{*}{\bf CGPs}} \\
& & {\bf programs} & & \\
\hline \hline
\multirow{3}{*}{{\bf Primal}}
& \emph{objective} & linear & linear & linear \\
\cline{2-5}
& {\multirow{2}{*}{\emph{constraints}}} & affine, & affine, sums of & affine, conic, \\
& & conic & exponentials & combined exponential/affine \\
\hline
\multirow{5}{*}{{\bf Dual}}
& {\multirow{3}{*}{\emph{objective}}} & {\multirow{3}{*}{linear}} & normalized & (negative) joint \\
& & & entropy & relative entropy \\
& & & + linear & + linear \\
\cline{2-5}
& {\multirow{2}{*}{\emph{constraints}}} & affine, & affine, & affine, \\
& & conic & nonnegativity & conic \\
\cline{2-5}
\hline
\end{tabular}
\caption{A summary of the types of objective functions and constraints in the primal and dual formulations of conic programs, GPs, and CGPs.  Note that the dual formulation consists of the maximization of a concave objective function subject to convex constraints.}\label{fig:gp-sdp-cgp}
\end{figure*}

\paragraph{Solutions of motivating problems} The enhanced expressive power of CGPs enables us to solve our motivating problems described earlier in the introduction.  CGPs can be used to approximately maximize the permanent over a family of (entrywise) nonnegative matrices that is specified via affine and conic constraints, e.g., a set that is defined as a section of the cone of positive semidefinite matrices.  We provide bounds on the quality of the approximation based on Van der Waerden's conjecture, which was originally settled in the 1980s \cite{Ego1981,Fal1981} and for which Gurvits gave an elegant proof more recently \cite{Gur2008}.  CGPs also provide solutions to GPs that are robust to parameter uncertainty when the uncertainty sets are specified via affine and conic constraints.  In this paper, we consider worst-case deterministic variability in the parameters. As our third example, we demonstrate that CGPs can compute the hitting times in linear dynamical systems consisting of multiple modes in which the initial point may belong to a set specified by affine and conic constraints, and in which the value of each mode is not known precisely (each mode may take on any value in some known range).  Finally, CGPs can be used to exactly or approximately compute certain capacities associated with quantum channels; as an example, in Section~\ref{subsec:quantumcapacity} we provide a comparison between a ``classical-to-quantum'' capacity of a quantum channel, and the capacity of a classical channel induced by the quantum channel (i.e., one is restricted to send and receive only classical information).  These and other applications of CGPs are detailed in Section~\ref{sec:applications}.


\paragraph{Paper outline} In Section~\ref{sec:defn}, we define a CGP, derive the corresponding dual problem of relative entropy minimization, and discuss the complexity of solving a CGP using an interior-point method.  We give examples of functions and constraints that can be expressed via CGPs in Section~\ref{sec:cgppower}.  We describe solutions to our motivating problems based on CGP in Section~\ref{sec:applications}.  We conclude with a brief summary and discussion of further research directions in Section~\ref{sec:discussion}.

\paragraph{Notational convention} The nonnegative orthant in $\R^k$ of vectors consisting of entrywise nonnegative entries is denoted by $\R^k_+$ and the positive orthant with strictly positive elements by $\R^k_{++}$.  The space of $n \times n$ symmetric matrices is denoted by $\Sym(n)$, which is a subspace of $\R^{n \times n}$, in particular, $\Sym(n) \cong \R^{n+1 \choose 2}$. We refer throughout the paper to (closed) \emph{convex cones}, which are sets that are closed under nonnegative linear combinations of the elements.  For a vector $\bx \in \R^n$ we denote the elementwise exponential by $\exp\{\bx\} \in \R^n$ with curly braces so that $[\exp\{\bx\}]_i = \exp(\bx_i)$.  Similarly the elementwise logarithm for vectors $\bx \in \R^n_{++}$ is denoted by $\log\{\bx\} \in \R^n$.  For vectors $\bx, \by \in \R^n$ we denote elementwise ordering by $\bx \leq \by$ to specify that $\bx_i \leq \by_i$ for $i = 1,\dots,n$.  We denote positive semidefinite ordering by $\preceq$ so that $\bM \preceq \bN$ implies $\bN - \bM$ is a positive semidefinite matrix for symmetric matrices $\bM, \bN \in \Sym(n)$.  The \emph{dual} of a convex cone $\K \subseteq \R^n$ is denoted $\K^{\star}$ and it is the set of elements in $\R^n$ that form a nonnegative inner-product with every element of $\K$:
\begin{equation}
\K^{\star} = \{ \bnu ~ | ~ \bx' \bnu \geq 0, ~ \forall \bx \in \K\}. \label{eq:conedual}
\end{equation}
Finally, for any real-valued function $f$ with domain $\mathcal{D} \subseteq \R^n$, the \emph{Fenchel} or \emph{convex conjugate} $f^\star$ is given by:
\begin{equation}
f^\star(\bnu) = \sup\{\bx' \bnu - f(\bx) ~ |~ \bx \in \mathcal{D}\} \label{eq:conjugate}
\end{equation}
for $\bnu \in \R^n$.  For more details of these concepts from convex analysis, we refer the reader to \cite{Roc1970}.


\section{Definition, Duality, and Numerical Aspects} \label{sec:defn}

In this section, we give a formal definition of CGPs and derive the associated dual problems.  We also describe the key elements of an interior-point method for solving CGPs.

\subsection{What is a CGP?} \label{subsec:cgp}
CGPs are optimization problems specified in terms of conic, affine, and combined exponential/affine constraints:
\begin{definition}
A \emph{CGP} is a convex optimization problem in a decision variable $\bx \in \R^n$ as follows:
\begin{equation}
\begin{aligned}
\inf_{\bx \in \R^n} & ~~~~~~~~~~~~~~~~~~~ \bp' \bx& \\ \mbox{s.t.} & ~~~ \exp\{\bQ' \bx - \bu\} + \bR' \bx - \bv \leq \bzero& \\ & ~~~~~~~~~~~~~~~~ \bS' \bx = \bw & \\ & ~~~~~~~~~~~~~~~~~~ \bx \in \K &
\end{aligned}\tag{P}\label{eq:cgpprimal}
\end{equation}
Here $\bQ, \bR \in \R^{n \times k}, \bS \in \R^{n \times m}$, and $\bu, \bv \in \R^k, \bw \in \R^m, \bp \in \R^n$ are fixed parameters.  Further, $\K$ is a convex cone.
\end{definition}

CGPs can be specialized to obtain conic optimization problems.  Specifically, by considering CGPs without the combined exponential/affine constraint, we obtain optimization problems in which a linear function is minimized subject to affine constraints and a conic constraint -- with appropriate choice of this cone, one can obtain conic programs such as LPs ($\K$ is the nonnegative orthant), SOCPs ($\K$ is the second-order or Lorentz cone), and SDPs ($\K$ is the cone of positive semidefinite matrices).  Similarly, CGPs can be specialized to obtain GPs as special cases.  In particular, we may neglect the conic constraint (i.e., $\K = \R^n$), and we need to be able to express constraints of the form:
\begin{equation*}
\sum_{i=1}^k \bc_i \exp([\bQ' \bx]_i) + d \leq 0
\end{equation*}
for $\bc \in \R^k_{++}$, $\bQ \in \R^{n \times k}$, and $d \in \R$.  Letting $\bu = \log\{\bc\} \in \R^k$ be the vector of the elementwise logarithms of the coefficients, this constraint is equivalent to the following system of constraints in decision varibles $\bx, \by \in \R^n$ and fixed parameters $\bQ \in \R^{n \times k}, \bu \in \R^k, d \in \R$:
\begin{equation*}
\exp\{\bQ' \bx + \bu\} - \by \leq \bzero; ~~~ \ones' \by + d \leq 0,
\end{equation*}
each of which can be expressed within the CGP framework \eqref{eq:cgpprimal}.  Here $\ones \in \R^n$ refers to the all-ones vector.  In this manner, CGPs provide a common generalization of LPs, SOCPs, SDPs, and GPs.

\subsection{Dual of a CGP} \label{subsec:dualcgp}
The dual of a CGP is specified in terms of the \emph{relative entropy} function of two vectors $\bnu, \blambda \in \R^k_{+}$:
\begin{equation}
D(\bnu, \blambda) = \sum_{i=1}^k \bnu_i \log\left(\frac{\bnu_i}{\blambda_i}\right). \label{eq:relent}
\end{equation}
We adopt the convention that $D(\mathbf{0},\mathbf{0}) = 0$.  The relative entropy function is \emph{jointly convex} in both the arguments.  In order to transparently obtain the convex dual of a CGP, we first establish a relationship between the relative entropy function and the combined exponential/affine function via convex conjugacy.  This result is well-known but we include it here for completeness.
\begin{lemma} \label{theo:explinconj}
Let $\chi_{\el}(\bx,\by): \R^n \times \R^n \rightarrow \R$ denote the characteristic function of the exponential/affine as follows:
\begin{equation*}
\chi_{\el}(\bx,\by) = \begin{cases} 0, & \mathrm{if} ~ \exp\{\bx\}+\by \leq \bzero \\ \infty, & \mathrm{otherwise} \end{cases}
\end{equation*}
Then the convex conjugate $\chi_{\el}^\star(\bnu,\blambda) = D(\bnu,e\blambda)$ with domain $(\bnu, \blambda) \in \R^n_+ \times \R^n_+$, where $e$ is the Euler constant.
\end{lemma}

\begin{proof}
From the definition \eqref{eq:conjugate} of the convex conjugate of a function we have that:
\begin{eqnarray}
\chi_{\el}^\star(\bnu,\blambda) &=& \sup_{\bx, \by} ~ \bx' \bnu + \by' \blambda ~ \mathrm{s.t.} ~ \exp\{\bx\} + \by \leq \bzero \nonumber \\ &=& \sup_{\bx} ~ \bx' \bnu - \exp\{\bx\}' \blambda. \label{eq:relentconj1}
\end{eqnarray}
Taking derivatives with respect to each $\bx_i$ and setting the result equal to $0$, we have that the optimal $\hat{\bx}$ must satisfy:
\begin{equation*}
\bnu_i - \exp(\hat{\bx}_i) \blambda_i = 0
\end{equation*}
for each $i = 1,\dots,n$.  Plugging in the optimal value $\hat{\bx}_i = \log\left(\tfrac{\bnu_i}{\blambda_i}\right)$ into \eqref{eq:relentconj1}, we have that
\begin{eqnarray*}
\chi_{\el}^\star(\bnu,\blambda) &=& \sum_i \bnu_i \log\left(\tfrac{\bnu_i}{\blambda_i}\right) - \bnu_i \\ &=& D(\bnu,e\blambda),
\end{eqnarray*}
which is the claimed result.
\end{proof}

We are now in a position to state the main result concerning the duals of CGPs.  As CGPs are specified in terms of exponential/affine functions, it is natural to expect the relative entropy function to play a prominent role in the dual based on the conjugacy result of Lemma~\ref{theo:explinconj}.
\begin{proposition} \label{theo:cgpdual}
The dual of the CGP \eqref{eq:cgpprimal} is the following convex optimization problem:
\begin{equation}
\begin{aligned}
\sup_{\bnu , \blambda \in \R^k_{+}, ~ \bdelta \in \R^m} & ~~~ -[D(\bnu, e \blambda) + \bu' \bnu + \bv' \blambda + \bw' \bdelta] & \\ \mbox{s.t.} & ~~~~~~~~ \bp + \bQ \bnu + \bR \blambda + \bS \bdelta \in \K^{\star} &
\end{aligned}\tag{D}\label{eq:dualcgp}
\end{equation}
Here $e$ is the Euler constant, $\K^{\star} \subseteq \R^n$ is the dual of the cone $\K$ \eqref{eq:conedual}, and $\bQ, \bR, \bS$ are the matrices specified in the primal form of a CGP \eqref{eq:cgpprimal}.
\end{proposition}
\begin{proof}
We consider a CGP in the following form with additional variables $\by, \bz \in \R^k$ for the exponential/affine constraints in order to simplify the derivation of the dual:
\begin{eqnarray*}
\inf_{\bx \in \R^n, \by, \bz \in \R^k} & \bp' \bx + \chi_{\el}(\by,\bz) & \\ \mbox{s.t.} & \bQ' \bx - \bu \leq \by & \\ & \bR' \bx - \bv \leq \bz  & \\ & \bS' \bx = \bw & \\ & \bx \in \K &
\end{eqnarray*}
Introducing dual variables $\bnu, \blambda \in \R^k_+, \bdelta \in \R^m,\bmu \in \K^{\star} \subseteq \R^n$, we have the following Lagrangian:
\begin{eqnarray*}
L(\bx,\by,\bz,\bnu,\blambda,\bdelta,\bmu) &=& \bx' \bp + \chi_{\el}(\by,\bz) + [\bx' \bQ - \bu' - \by'] \bnu \\ && + [\bx' \bR - \bv' - \bz'] \blambda + [\bx' \bS - \bw'] \bdelta - \bx' \bmu \\ &=& \chi_{\el}(\by,\bz) - \by' \bnu - \bz' \blambda \\ && + \bx' [\bp + \bQ \bnu + \bR \blambda + \bS \bdelta - \bmu] - \bu' \bnu - \bv' \blambda - \bw' \bdelta,
\end{eqnarray*}
where the second expression is obtained by a rearrangement of terms.  Minimizing the Lagrangian over the primal variables $\by,\bz$, we have the following simplified expression in terms of the joint relative entropy function by appealing to Lemma~\ref{theo:explinconj}:
\begin{equation}
\inf_{\by,\bz} L(\bx,\by,\bz,\bnu,\blambda,\bmu) = -D(\bnu,e\blambda) + \bx' [\bp + \bQ \bnu + \bR \blambda + \bS \bdelta - \bmu] - \bu' \bnu - \bv' \blambda - \bw' \bdelta. \label{eq:marglagrangian}
\end{equation}
Next we minimize over the primal variable $\bx$ and obtain the following condition that must be satisfied by the dual variables:
\begin{equation*}
\bp + \bR \blambda + \bQ \bnu + \bS \bdelta -\bmu = 0.
\end{equation*}
By eliminating the dual variable $\bmu$, we are left with the following constraint:
\begin{equation*}
\bp + \bR \blambda + \bQ \bnu + \bS \bdelta \in \K^{\star}.
\end{equation*}
From \eqref{eq:marglagrangian} and by incorporating this constraint, we obtain the dual optimization problem \eqref{eq:dualcgp}.
\end{proof}


We emphasize that the dual of a CGP is a \emph{joint} relative entropy optimization problem in which the negative of the relative entropy between two vectors is jointly maximized subject to linear and conic constraints on the vectors.  The dual of a standard GP is a \emph{normalized entropy} maximization problem subject to linear and nonnegativity constraints (see \cite{BoyKVH2007,Chi2005} for an extensive survey of GPs).  Specifically, the normalized entropy of a nonnegative vector $\bnu \in \R^k_+$ is a concave function defined as follows \cite{BoyV2004}:
\begin{equation*}
h_{\mathrm{norm}}(\bnu) = -\sum_{i=1}^k \bnu_i \log\left(\frac{\bnu_i}{\ones' \bnu}\right)
\end{equation*}
This function can be viewed as a restriction of the negative of the joint relative entropy function \eqref{eq:relent} in which each entry of the second argument $\blambda$ is identically equal to $\blambda_i = \ones' \bnu$. Therefore, the additional expressive power provided by CGPs in contrast to standard GPs can be understood from the dual viewpoint as the result of two generalizations -- the first is the negative relative entropy function in the objective instead of the normalized entropy, and the second is the linear and general conic constraints rather than just linear and nonnegativity constraints.  The dual of a CGP can also be viewed as a generalization of the dual of a conic program: although the constraints in both cases are of the same type, the dual of a conic program consists of a linear objective while the dual of a CGP consists of a joint relative entropy term as well as linear terms.  We refer the reader to the table in Figure~\ref{fig:gp-sdp-cgp} for a concise summary.

\subsection{Numerical Solution of CGPs} \label{subsec:cgpsolver}

Interior-point techniques \cite{NesN1994} provide a general methodology via \emph{self-concordant barrier functions} to obtain polynomial-time algorithms for general families of convex programs such as LPs, SOCPs, GPs, and SDPs.  We refer the reader to many excellent texts for details \cite{BenN2001,NesN1994,NocW1999,Ren1997}.  Here we highlight the types of barrier functions that may be employed in interior-point methods for CGPs, and we give runtime bounds for these algorithms.

For conic constraints, natural self-concordant barriers for cones such as the orthant, the second-order cone, and the semidefinite cone are specified by the associated \emph{logarithmic barriers} \cite{NesN1994,NesT1998}.  The situation is a bit more subtle for the combined exponential/affine constraints $\exp\{\bQ' \bx - \bu\} + \bR' \bx - \bv \leq \bzero$ in the primal form of a CGP (with $\bQ, \bR \in \R^{n \times k}$ and $\bu,\bv \in \R^k$).  Consider the following restatement of this constraint with additional variables $\by, \bz \in \R^k$:
\begin{equation*}
\by - \log\{-\bz\} \leq \bzero, ~ \bz \leq \bzero, ~ \bQ' \bx - \bu \leq \by, ~ \bR' \bx - \bv \leq \bz.
\end{equation*}
This method of rewriting exponential constraints is also employed in interior-point solvers for GPs \cite{NesN1994}.  The redundancy in these constraints is useful because the following logarithmic barrier function for these constraints is self-concordant \cite{NesN1994}:
\begin{equation*}
\sum_{i=1}^k \Big(\log(-\bz_i) + \log\big(\log(-\bz_i) - \by_i \big) + \log\big(\by_i +\bu_i - [\bQ' \bx]_i \big) + \log\big(\bz_i +\bv_i - [\bR' \bx]_i \big) \Big).
\end{equation*}
One can directly write down the usual logarithmic barrier for the constraint $\exp\{\bQ' \bx - \bu\} + \bR' \bx - \bv \leq \bzero$ without the additional variables $\by,\bz$, but this barrier is not known to be self-concordant.  The runtime analysis of interior-point procedures relies on the self-concordance assumption, and therefore one resorts to the above reformulation \cite{NesN1994}.  An upper bound on the number of Newton steps that must be computed to obtain an $\epsilon$-accurate solution of a CGP is given by \cite{NesN1994}:
\begin{equation*}
\# ~ \mathrm{Newton ~ steps} ~ = ~ \mathcal{O}\left(\sqrt{k + \vartheta_\K} \log\left(\tfrac{k + \vartheta_\K}{\epsilon}\right)\right).
\end{equation*}
Here $\vartheta_\K$ is a ``complexity parameter'' of the barrier associated with the cone $\K$ \cite{NesN1994}.  If $\K = \R^n_+$ is the nonnegative orthant in the primal CGP \eqref{eq:cgpprimal}, then $\vartheta_\K = n$ and we have the following bound:
\begin{equation*}
\# ~ \mathrm{total ~ operations ~ for ~ CGP ~ with ~ orthant ~ constraint} ~ = ~ \mathcal{O}\left((k + n)^{3.5} \log\left(\tfrac{k + n}{\epsilon}\right) \right).
\end{equation*}
Similarly, if $\K \subset \Sym(\ell)$ is the cone of positive semidefinite matrices in \eqref{eq:cgpprimal} (note that $n = {\ell+1 \choose 2}$), then $\vartheta_\K = \ell$ and we have:
\begin{equation*}
\# ~ \mathrm{total ~ operations ~ for ~ CGP ~ with ~ semidefinite ~ constraint} ~ = ~ \mathcal{O}\left((k + \ell)^{3.5} \ell^3 \log\left(\tfrac{k + \ell}{\epsilon}\right) \right).
\end{equation*}
Note that these bounds hold for arbitrary CGPs.  The types of CGPs that arise in practical settings frequently possess additional structure such as sparsity or symmetry, which can be exploited to significantly reduce the overall runtime.

\section{The Expressive Power of CGPs} \label{sec:cgppower}


A central question with any structured family of convex programs is the class of convex sets that can be represented within that family.  For example, all polytopes are representable within the framework of LPs.  The problem of characterizing the set of SDP-representable convex sets is more challenging and it is an active research topic \cite{BlePT2013,GouPT2013,HelN2009}.  As we have seen in the previous section, CGPs can be specialized to obtain traditional GPs as well as conic optimization problems such as SOCPs and SDPs.  In this section, we describe the expressive power of the CGP formulation in representing convex sets and functions that cannot be represented purely via GPs or via SDPs.  The examples discussed below play a prominent role in Section~\ref{sec:applications} where we demonstrate the utility of CGPs in a range of problem domains.  We denote the sum of the top-$r$ eigenvalues of $\bM \in \Sym(n)$ by $s_r(\bM)$.  In this section, we will exploit the fact that the sublevel sets of this function are SDP-representable \cite{BenN2001}:
\begin{eqnarray}
&s_r(\bM) \leq t& \nonumber \\ &\Updownarrow& \nonumber \\ &\exists ~ \bZ \in \Sym(n), ~ \alpha \in \R ~~ \mathrm{s.t.} ~~ \bZ \succeq \bzero, ~ \bM - \bZ \preceq \alpha \mathbf{I}, ~ r \alpha + \mathrm{trace}(\bZ) \leq t.& \label{eq:topreig}
\end{eqnarray}

\subsection{Basics of Lift-and-Project Representations} \label{subsec:liftandproj}


At an elementary level, the collection of CGP-representable sets in $\R^n$ includes those that can be represented as the intersection of sets specified by combined exponential/affine inequalities, affine constraints, and conic constraints, all in a decision variable $\bx \in \R^n$.  However, such direct CGP representations in $\R^n$ can be somewhat limited in their expressive power, and therefore it is of interest to consider \emph{lifted} CGP representations in which a convex set in $\R^n$ is defined as the \emph{projection} of a CGP-representable set in $\R^{n'}$ for $n' > n$.  Such lift-and-project CGP representations are extremely useful if $n'$ is not much larger than $n$ and if the CGP representation in $\R^{n'}$ is tractable to compute.  Here we give a simple example of a lift-and-project representation to illustrate its power.  The remainder of Section~\ref{sec:cgppower} is devoted to more sophisticated examples.

The logarithm function has the following variational representation for $\alpha > 0$, which can be easily verified:
\begin{equation}
\log(\alpha) = \inf_{\beta \in \R} ~ e^{-\beta-1} \alpha + \beta. \label{eq:logliftproj}
\end{equation}
Here $\beta$ plays the role of the additional variable or ``lifting dimension.''  The appearance of $\beta$ in this nonlinear expression with exponential and affine terms is somewhat non-traditional; in comparison, in conic lift-and-project representations the additional variables appear in linear and conic constraints.  This simple observation yields a number of interesting CGP lift-and-project representations.

\paragraph{Logarithm of sum-of-exponentials} Expressions involving a combination of an affine function, a sum-of-exponentials, and a logarithm of a sum-of-exponentials can be upper-bounded by appealing to \eqref{eq:logliftproj}:
\begin{eqnarray*}
&\log\left(\sum_{i=1}^k \bc_i \exp([\bQ' \bx]_i)\right) + \sum_{j=1}^{\bar{k}} \bar{\bc}_j \exp([\bar{\bQ}' \bx]_j) + \br' \bx - t \leq 0& \\ &\Updownarrow& \\ \exists \beta ~~ \mathrm{s.t.} ~ & \sum_{i=1}^k \bc_i \exp([\bQ' \bx]_i - \beta - 1) + \sum_{j=1}^{\bar{k}} \bar{\bc}_j \exp([\bar{\bQ}' \bx]_j) + \br' \bx - t + \beta \leq 0&
\end{eqnarray*}
Here $\bQ \in \R^{n \times k}, \bar{\bQ} \in \R^{n \times \bar{k}}, \bc \in \R^k_{+}, \bar{\bc} \in \R^{\bar{k}}_{+}, \br \in \R^n, t \in \R$ are fixed parameters.  The latter constraint can be specified in the CGP framework \eqref{eq:cgpprimal}, with the additional variable $\beta$ playing the role of the lifting dimension.  We note that upper bounds on a pure sum-of-exponentials function, or equivalently, on a logarithm of a sum-of-exponentials are considered within the framework of GPs \cite{BoyKVH2007,Chi2005}.  However, CGPs can be used to bound composite expressions involving both these types of functions based on the additional flexibility provided by constraints on combined exponential/affine functions.

\paragraph{Product of sum-of-exponentials} The reformulation \eqref{eq:logliftproj} of the logarithm is also useful in providing CGP representations of \emph{products} of expressions, each of which is sum-of-exponentials.  Specifically, consider the following expressions in a decision variable $\bx \in \R^n$ for $j = 1,\dots,p$:
\begin{equation}
f_j(\bx) = \sum_{i=1}^{k_j} \bc_i^{(j)} \exp \left( \left[{\bQ^{(j)}}' \bx\right]_i \right). \label{eq:prodsumexp1}
\end{equation}
As usual, $\bc^{(j)} \in \R^{k_j}_+, \bQ^{(j)} \in \R^{n \times k_j}$ for $j = 1, \dots, p$.  Suppose we wish to represent the constraint:
\begin{equation}
\prod_{j=1}^p f_j(\bx) \leq t. \label{eq:prodsumexp2}
\end{equation}
This expression can equivalently be written as:
\begin{equation*}
\sum_{j=1}^p \log(f_j(\bx)) \leq \log(t).
\end{equation*}
By appealing to \eqref{eq:logliftproj} and \eqref{eq:prodsumexp1}, we have that the constraint \eqref{eq:prodsumexp2} is equivalent to:
\begin{equation}
\exists \bbeta \in \R^p, \delta \in \R ~ \mathrm{s.t.} ~ \sum_{j=1}^p \sum_{i=1}^{k_j} \bc^{(j)}_i \exp \left( \left[{\bQ^{(j)}}' \bx\right]_i -\bbeta_j - 1\right) + \boldsymbol{\ones}' \bbeta \leq \delta ~ \mathrm{and} ~ \exp(\delta) \leq t. \label{eq:prodsumexp3}
\end{equation}
The reason we introduce the extra variable $\delta$ is to allow for potentially additional convex constraints on $t$, as constraints of the form \eqref{eq:prodsumexp2} may be embedded inside larger convex programs in which $t$ may be a decision variable.  Here again $\bbeta \in \R^p$ plays the role of the lifting dimension.  Using these techniques one can also bound composite expressions involving a sum of exponentials, a logarithm of a sum of exponentials, and a product of sums of exponentials.

\subsection{Sums of Exponentials with General Coefficients} \label{subsec:sumexpcgp}

As our next collection of examples, we consider constraints on sums of exponentials in which the coefficients may be more general than positive scalars.

\paragraph{Matrix-weighted sum-of-exponentials} Consider the following constraint in decision variable $\bx \in \R^n$ with $\bQ \in \R^{n \times k}$, $\bC^{(i)} \in \Sym(\ell)$ for $i= 1,\dots,k$, $\bR^{(j)} \in \Sym(\ell)$ for $j = 1,\dots,n$, $\bS \in \Sym(\ell)$, and where each $\bC^{(i)} \succeq \bzero$:
\begin{equation*}
\sum_{i=1}^k \bC^{(i)} \exp([\bQ' \bx]_i) + \sum_{j=1}^n \bR^{(j)} \bx_j + \bS \preceq \bzero.
\end{equation*}
This constraint can equivalently be written in terms of decision variables $\bx \in \R^n$ and $\by \in \R^k$ as:
\begin{equation*}
\sum_{i=1}^k \bC^{(i)} \by_i + \sum_{j=1}^n \bR^{(j)} \bx_j + \bS \preceq \bzero ; ~~~ \exp\{\bQ' \bx\} - \by \leq \bzero.
\end{equation*}
These latter constraints can be expressed within the framework of CGP.  The first constraint here is a \emph{linear matrix inequality} and an SDP consists purely of such types of constraints.  A CGP allows for additional exponential/affine constraints of the second type, thus broadening the scope of SDPs.  As an instance of this construction, one can impose upper-bound constraints on sums of the top-$r$ eigenvalues of matrix-valued expressions such as $\sum_{i=1}^k \bC^{(i)} \exp([\bQ' \bx]_i)$ with $\bC^{(i)} \succeq \bzero$ by applying \eqref{eq:topreig}.

\paragraph{Coefficients from general cones} Generalizing the previous example, we recall that the cone of positive semidefinite matrices is a \emph{proper} cone \cite{BenN2001} -- these are cones in Euclidean space that are closed, convex, have a non-empty interior, and contain no lines.  Other examples include the nonnegative orthant, the second-order cone, the cone of polynomials that can be expressed as sums-of-squares of polynomials of a given degree, and the cone of everywhere nonnegative polynomials of a given degree.  An important feature of proper cones is that they induce a \emph{partial ordering} analogous to the positive semidefinite ordering.  Specifically, for a proper cone $\K \subset \R^n$, we denote the corresponding partial order by $\preceq_\K$ so that $\bu \preceq_\K \bv$ if and only if $\bv - \bu \in \K$.  As with the previous example, we consider the following convex constraint in a decision variable $\bx \in \R^n$ with $\bQ \in \R^{n \times k}$, $\bC^{(i)} \in \R^\ell$ for $i= 1,\dots,k$, $\bR^{(j)} \in \R^\ell$ for $j = 1,\dots,n$, $\bS \in \R^\ell$, and where each $\bC^{(i)} \in \K$ for a proper cone $\K \subset \R^\ell$:
\begin{equation*}
\sum_{i=1}^k \bC^{(i)} \exp([\bQ' \bx]_i) + \sum_{j=1}^n \bR^{(j)} \bx_j + \bS \preceq_\K \bzero.
\end{equation*}
Again, introducing additional variables $\by \in \R^\ell$ we can rewrite this constraint as follows:
\begin{equation*}
\sum_{i=1}^k \bC^{(i)} \by_i + \sum_{j=1}^n \bR^{(j)} \bx_j +\bS \preceq_\K \bzero; ~~~ \exp\{\bQ' \bx\} \leq \by.
\end{equation*}
If the cone $\K$ is tractable to represent so that checking membership in the cone can be accomplished efficiently, then these constraints can be tractably specified via CGP.

\paragraph{Inverses via Schur complement} Based on the first example with matrix-weighted sums of exponentials, we now show that CGPs can also be used to express certain constraints involving inverses based on the Schur complement.  Consider the following matrix-valued expressions in a variable $\bx \in \R^n$:
\begin{eqnarray*}
F_{11}(\bx) &=& \sum_{i=1}^k \bC^{(i)}_{11} \exp([\bQ_{11}' \bx]_i) + \sum_{j=1}^n \bR^{(j)}_{11} \bx_j + \bS_{11} \\
F_{22}(\bx) &=& \sum_{i=1}^{\tilde{k}} \bC^{(i)}_{22} \exp([\bQ_{22}' \bx]_i) + \sum_{j=1}^n \bR^{(j)}_{22} \bx_j + \bS_{22} \\
F_{12}(\bx) &=& \sum_{j=1}^n \bR^{(j)}_{12} \bx_j + \bS_{12}.
\end{eqnarray*}
As in the first example with matrix coefficients, we have that $\bQ_{11} \in \R^{n \times k}$, $\bC^{(i)}_{11} \in \Sym(\ell)$ for $i= 1,\dots,k$; $\bQ_{22} \in \R^{n \times \tilde{k}}$, $\bC^{(i)}_{22} \in \Sym(\tilde{\ell})$ for $i= 1,\dots,\tilde{k}$; $\bR^{(j)}_{11} \in \Sym(\ell),\bR^{(j)}_{22} \in \Sym(\tilde{\ell}),\bR^{(j)}_{12} \in \R^{\ell \times \tilde{\ell}}$ for $j = 1,\dots,n$; $\bS_{11} \in \Sym(\ell), \bS_{22} \in \Sym(\tilde{\ell}), \bS_{12} \in \R^{\ell \times \tilde{\ell}}$; and where each $\bC^{(i)}_{11} \succeq \bzero$ and each $\bC^{(i)}_{22} \succeq \bzero$.  Then the following set of constraints can be expressed via CGP:
\begin{equation*}
F_{11}(\bx) \prec \bzero ~~~ \mathrm{and} ~~~ F_{22}(\bx) - F_{12}(\bx) F_{11}(\bx)^{-1} F_{12}(\bx)' \prec \bzero.
\end{equation*}
In particular, this set of constraints is equivalent based on the Schur complement to the single constraint:
\begin{equation*}
\left[ \begin{array}{cc}
F_{11}(\bx) & F_{12}(\bx) \\
F_{12}'(\bx) & F_{22}(\bx) \end{array} \right] \prec \bzero
\end{equation*}
This constraint on a symmetric matrix expression of size $(\ell + \tilde{\ell}) \times (\ell + \tilde{\ell})$ can be specified using CGP by following the first example on matrix-weighted sums of exponentials.  The Schur complement method allows us to bound expressions with inverses.  For instance, consider the following set of constraints for $t > 0$:
\begin{equation*}
\sum_{j=1}^k \frac{\left({\br^{(j)}}' \bx - \bs_j \right)^2}{-f_j(\bx)} < t; ~~~ f_j(\bx) < 0,
\end{equation*}
where each $f_j(\bx)$ may be an expression involving a sum of exponentials, affine functions, products of sums of exponentials, and logarithms of sums of exponentials (these latter two types of expressions may be transformed to combined exponential/affine from Section~\ref{subsec:liftandproj}).  Here each $\br^{(j)} \in \R^n$ and $\bs \in \R^k$.  One can express this system of constraints within the CGP framework by applying the Schur complement procedure outlined above.


\subsection{Maximum over Sums of Exponentials} \label{subsec:maxsumexpcgp}

CGPs provide a flexible methodology to deal with constraints on the maximum over a collection of sums of exponentials.  Such constraints arise prominently in robust GP in which one seeks solutions that are robust to uncertainty in the parameters of a GP.  These types of constraints also arise in applications ranging from permanent maximization to hitting-time estimation in dynamical systems -- see Section~\ref{sec:applications} for details.  Formally, consider a maximum over a set of sums of exponentials in a decision variable $\bx \in \R^n$ of the following type:
\begin{eqnarray}
\sup_{\bc \in \cC, \bq^{(i)} \in \cQ^{(i)}} & \sum_{i=1}^k \bc_i \exp \left( {\bq^{(i)}}' \bx \right) \leq t& \label{eq:robustgp} \\ &\Updownarrow& \nonumber \\ \forall \bc \in \cC, \bq^{(i)} \in \cQ^{(i)} & \sum_{i=1}^k \bc_i \exp\left({\bq^{(i)}}' \bx \right) \leq t & \nonumber
\end{eqnarray}
Here $\cC \subset \R^k_+$ and $\cQ^{(i)} \subset \R^n$ for $i = 1, \dots, k$, and in the robust optimization literature these are usually referred to as parameter uncertainty sets.  In principle, these constraints specify convex sets because they can be viewed as the intersection of a (possibly infinite) collection of convex constraints.  When $\cC$ and each $\cQ^{(i)}$ are finite sets, then constraints of the form \eqref{eq:robustgp} reduce to a finite collection of constraints on sums of exponentials.  For the case in which $\cC$ and each $\cQ^{(i)}$ are infinite, the challenge is to obtain a tractable representation for the corresponding constraint set via a small, finite number of constraints.  The literature on robust optimization \cite{BenEN2009,BenN1998} has considered this question in great detail in a number of settings; in Section~\ref{subsec:robustgp} we discuss the previous state-of-the-art results for handling parameter uncertainty in sum-of-exponentials functions in the context of robust GPs, and we contrast these with the significant advance facilitated by CGPs.

In order to obtain tractable representations for constraints of the form \eqref{eq:robustgp}, we focus here on sets $\cC$ and $\cQ^{(i)}$ that are themselves tractable to represent.  If these sets are convex, and in particular they have tractable conic representations, then CGPs provide a natural framework to obtain efficient representations for \eqref{eq:robustgp}.  The next result illuminates this point by appealing to convex duality:

\begin{proposition} \label{theo:robustgp}
Suppose $\cC \subset \R^k$ is a convex set of the form:
\begin{equation*}
\cC = \{\bc ~|~ \bc \geq \bzero, ~ \bg + \bF \bc \in \K \}
\end{equation*}
for a convex cone $\K \subset \R^\ell$ and $\bg \in \R^\ell, \bF \in \R^{\ell \times k}$, and $\cQ^{(i)} \subset \R^n$ for $i = 1,\dots,k$ are each sets of the form:
\begin{equation*}
\cQ^{(i)} = \left\{\bq^{(i)} ~|~ \bh^{(i)} + \bM^{(i)} \bq^{(i)} \in \cF^{(i)} \right\}
\end{equation*}
for convex cones $\cF^{(i)} \subset \R^m$ and $\bh^{(i)} \in \R^m, \bM^{(i)} \in \R^{m \times n}$.  Further, assume that there exists a point in $\cC$ satisfying the conic and nonnegativity constraints strictly, and similarly that there exists a point for each $\cQ^{(i)}$ that satisfies the corresponding conic constraint strictly.  Then we have that $\bx \in \R^n$ satisfies the constraint \eqref{eq:robustgp} if and only if there exist $\bzeta \in \R^\ell, ~ \btheta^{(i)} \in \R^m ~\mathrm{for}~ i = 1,\dots,k$ such that:
\begin{eqnarray*}
&\bg' \bzeta \leq t, ~ \bzeta \in \K^\star,& \\ &{\bM^{(i)}}' \btheta^{(i)} + \bx = \bzero, ~ \exp\left({\bh^{(i)}}' \btheta^{(i)} \right) + [\bF' \bzeta]_i \leq 0, ~ \btheta^{(i)} \in {\cF^{(i)}}^\star, ~\mathrm{for} ~ i = 1,\dots,k&
\end{eqnarray*}
\end{proposition}


\noindent \textbf{Note}: The assumption that there exist points that strictly satisfy the constraints specifying $\cC$ and those specifying each $\cQ^{(i)}$ allows us to appeal to strong duality in deriving our result \cite{Roc1970}.

\begin{proof}
The constraint \eqref{eq:robustgp} can equivalently be written as follows:
\begin{eqnarray*}
\exists \by \in \R^k ~ \mathrm{s.t.}~ & \sup_{\bc \in \cC} \by' \bc \leq t & \\ & \sup_{\bq^{(i)} \in \cQ^{(i)}} \exp\left({\bq^{(i)}}' \bx\right) \leq \by_i,& ~ \forall i = 1,\dots,k.
\end{eqnarray*}
This restatement is possible because the set $\cC$ is a subset of the nonnegative orthant $\R^k_+$, and because the uncertainty sets $\cQ^{(i)}$ are decoupled and therefore independent of each other.  The first expression, $\sup_{\bc \in \cC} ~ \by' \bc ~ \leq ~ t$, is a universal quantification for all $\bc \in \cC$. In order to convert this universal quantifier to an existential quantifier, we appeal to convex duality as is commonly done in the theory of robust optimization \cite{BenEN2009,BenN1998}.  Specifically, by noting that $\cC$ has a conic representation, we have that:
\begin{eqnarray}
\sup_{\bc \geq \bzero, \bg + \bF \bc \in \K} ~ &\by' \bc \leq t& \nonumber \\ &\Updownarrow& \nonumber \\  \exists \bzeta \in \R^\ell ~ \mathrm{s.t.} ~ &\bF' \bzeta + \by \leq \bzero, ~ \bzeta \in \K^\star, ~ \bg' \bzeta \leq t.& \label{eq:robustgp1}
\end{eqnarray}
One can obtain this result in a straightforward manner via conic duality \cite{BenN2001}.  Similarly, we have that:
\begin{eqnarray}
\sup_{\bq^{(i)} \in \cQ^{(i)}} ~ &\exp\left({\bq^{(i)}}' \bx\right) \leq \by_i& \nonumber \\ &\Updownarrow& \nonumber \\  \exists \btheta^{(i)} \in \R^m ~ \mathrm{s.t.} ~ &{\bM^{(i)}}' \btheta^{(i)} + \bx = 0, ~ \btheta^{(i)} \in {\cF^{(i)}}^\star, ~ \exp\left({\bh^{(i)}}' \btheta^{(i)} \right) \leq \by_i.& \label{eq:robustgp2}
\end{eqnarray}
The assumptions on strict feasibility are required to derive \eqref{eq:robustgp1} and \eqref{eq:robustgp2}.  Combining these results and eliminating $\by$, we have the desired conclusion.
\end{proof}

In summary, Proposition~\ref{theo:robustgp} gives a lifted representation for \eqref{eq:robustgp} with additional variables, and with constraints that can be specified within the CGP framework \eqref{eq:cgpprimal}.  As with expressions involving logarithms \eqref{eq:logliftproj}, the appearance of the additional variables $\bzeta$ and $\btheta^{(i)}$ in nonlinear expressions involving exponential and affine terms provides a powerful technique to obtain tractable CGP lift-and-project representations.  Hence, bounds on sum-of-exponentials functions with parameter uncertainty of the form \eqref{eq:robustgp} in which the uncertainty sets have tractable conic representations can be efficiently specified via CGPs.

\paragraph{Maximum over logarithms, products of sum-of-exponentials} Proposition~\ref{theo:robustgp} is also useful for providing CGP representations of the maximum of a composite expression involving a sum of exponentials, a product of sums of exponentials, and a logarithm of sums of exponentials.  By appealing to the techniques of Section~\ref{subsec:liftandproj}, such a composite expression may be transformed to an expression formed only of a sum of exponentials and an affine function.  As a result, uncertainty in the parameters of the exponentials of the original expression with products and logarithms can directly be interpreted as uncertainty in the parameters of the exponentials of the final expression involving only a sum of exponentials and an affine function.  Consequently, Proposition~\ref{theo:robustgp} can be applied to this final expression, thus leading to a tractable formulation for parameter uncertainty in expressions with products and logarithms of sums of exponentials.

\subsection{CGP-Representable Spectral Functions} \label{subsec:spectralcgp}

An elegant result of Lewis \cite{Lew1995}, extending earlier work by Von Neumann \cite{Von1937}, states that any convex function of a symmetric matrix that is invariant under conjugation by orthogonal matrices is a convex and permutation-invariant function of the eigenvalues.  This result along with the next proposition provides a recipe for CGP representations of certain convex sets in the space of symmetric matrices.

\begin{proposition}\cite{BenN2001} \label{theo:convexspectral}
Let $f : \R^n \rightarrow \R$ be a convex function that is invariant under permutation of its argument, and let $g: \Sym(n) \rightarrow \R$ be the convex function defined as $g(\bM) = f(\lambda(\bM))$.  Here $\lambda(\bM)$ refers to the list of eigenvalues of $M$.  Then we have that
\begin{eqnarray*}
&g(\bM) \leq t& \\ &\Updownarrow& \\ \exists \bx \in \R^n ~ \mathrm{s.t.} ~ &f(\bx) \leq t& \\ &\bx_1 \geq \cdots \geq \bx_n& \\ &s_r(\bM) \leq \bx_1 + \cdots + \bx_r, ~ r = 1,\dots,n-1& \\ &\mathrm{trace}(\bM) = \bx_1 + \cdots + \bx_n.&
\end{eqnarray*}
\end{proposition}

By appealing to the fact that the function $s_r$ is SDP-representable from \eqref{eq:topreig}, we obtain the next example of a CGP-representable set in the space of symmetric matrices.

\paragraph{GP-representable functions of spectra} Consider any convex and permutation-invariant function $f: \R^n \rightarrow \R$ that is representable via CGP.  Then by appealing to Proposition~\ref{theo:convexspectral}, we have that the function $g: \Sym(n) \rightarrow \R$ defined as $g(\bM) = f(\lambda(\bM))$ is also representable via CGP.  A prominent special case of this class of CGP-representable functions is the Von-Neumann entropy of a matrix $\bM \in \Sym(n)$ with $\bM \succeq \bzero$:
\begin{equation}
H_{vn}(\bM) \triangleq -\mathrm{trace}[\bM \log(\bM)] = -\sum_{i=1}^n \lambda_i(\bM) \log(\lambda_i(\bM)) = -D(\lambda(\bM), \ones). \label{eq:vonneumann}
\end{equation}
Here $\lambda_i(\bM)$ refers to the $i$'th largest eigenvalue of $\bM$.  In quantum information theory, the Von-Neumann entropy is typically considered for positive semidefinite matrices with trace equal to one (i.e., quantum density matrices), but the function itself is concave over the full cone of positive semidefinite matrices.  As Von-Neumann entropy is concave and invariant under conjugation of the argument by any orthogonal matrix (i.e., it is a concave and permutation-invariant function of the eigenvalues of its argument), we can appeal to Proposition~\ref{theo:convexspectral} to obtain a CGP representation.  Specifically, as $-H_{vn}(\bM) = D(\lambda(\bM),\ones)$, sublevel sets of negative Von-Neumann entropy are representable via (the dual of a) CGP.

\section{Applications} \label{sec:applications}

In the preceding sections we described the conceptual unification of GPs and of conic optimization problems such as SDPs provided by the CGP framework.  In this section, we discuss the utility of CGPs in solving problems in a range of application domains; these problems fall outside the scope of GPs and of SDPs due to the limitatons of each of these frameworks, as discussed in the introduction.  All of the numerical experiments in this section employed the CVX modeling framework \cite{GraB2008} and the SDPT3 solver \cite{TutTT2003}.

\subsection{Permanent Maximization} \label{subsec:permmax}
The permanent of a matrix $\bM \in \R^{n \times n}$ is defined as:
\begin{equation}
\mathrm{perm}(\bM) \triangleq \sum_{\sigma \in \mathfrak{S}_n} \prod_{i=1}^n \bM_{i,\sigma(i)}, \label{eq:perm}
\end{equation}
where $\mathfrak{S}_n$ refers to the set of all permutations of the ordered set $[1,\dots,n]$.  The permanent of a matrix is a quantity that arises in combinatorics as the sum over weighted perfect matchings in a bipartite graph \cite{Min1984}, in geometry as the mixed volume of hyperrectangles \cite{Bet1992}, and in multivariate statistics in the computation of order statistics \cite{BapB1989}.  In this section we consider the problem of maximizing the permanent over a family of matrices.  This problem has received some attention for families of positive semidefinite matrices with specified eigenvalues \cite{DreJ1989,GroJEW1986}, but all of these works have tended to seek analytical solutions for very special instances of this problem.  Here we consider permanent maximization over general convex subsets of the set of nonnegative matrices.  (Recall that SDPs are useful for maximizing the \emph{determinant} of a symmetric positive definite matrix subject to affine and conic constraints \cite{VanBW1998}.)

\begin{problem}
Given a family of matrices $\cC \in \R^{n \times n}$, find a matrix in $\cC$ with maximum permanent.
\end{problem}

Permanent maximization is relevant for designing a bipartite network in which the average weighted matching is to be maximized subject to additional topology constraints on the network.  The question also arises in geometry in finding a configuration of a set of hyperrectangles so that their mixed volume is maximized.  Even computing the permanent of a fixed matrix is a $\#$P-hard problem and it is therefore considered to be intractable in general; accordingly a large body of research has investigated approximate computation of the permanent \cite{Bar1997,GurS2002,JerSV2004,LinSW2000}.  The class of elementwise nonnegative matrices has particularly received attention as these most commonly arise in application domains (e.g., bipartite graphs with nonnegative edge weights).  For such matrices, several \emph{deterministic} polynomial-time approximation algorithms provide exponential-factor (in the size $n$ of the matrix) approximations of the permanent, e.g. \cite{LinSW2000}.  The approach in \cite{LinSW2000} describes a technique based on solving a GP to approximate the permanent of a nonnegative matrix.  The approximation guarantees in \cite{LinSW2000} are based on the Van Der Waerden conjecture, which states that the matrix $\bM = \tfrac{1}{n} \ones \ones'$ has the strictly smallest permanent among all $n \times n$ doubly stochastic matrices (nonnegative matrices with all row and column sums equal to one).  This conjecture was proved originally in the 1980s \cite{Ego1981,Fal1981}, and Gurvits recently gave a very simple and elegant proof of this result \cite{Gur2008}.  In what follows, we describe the approach of \cite{LinSW2000}, but using the terminology developed by Gurvits in \cite{Gur2008}.

The permanent of a matrix $\bM \in \R^{n \times n}_+$ can be defined in terms of a particular coefficient of the following homogenous polynomial in $\by \in \R^n_+$:
\begin{equation}
p_\bM(\by_1,\dots,\by_n) = \prod_{i=1}^n \left(\sum_{j=1}^n \bM_{i,j} \by_j \right). \label{eq:permpoly}
\end{equation}
Specifically, the permanent of $\bM$ is the coefficient corresponding to the $\by_1 \cdots \by_n$ monomial term of $p_\bM(\by_1,\dots,\by_n)$:
\begin{equation*}
\mathrm{perm}(\bM) = \frac{\partial^n}{\partial \by_1 \cdots \partial \by_n} p_\bM(\by_1,\dots,\by_n)
\end{equation*}
In his proof of Van Der Waerden's conjecture, Gurvits defines the \emph{capacity} of a homogenous polynomial $p(\by_1,\dots,\by_n)$ of degree $n$ defined over $\by \in \R^n_+$ as follows:
\begin{eqnarray}
\mathrm{cap}(p) &\triangleq& \inf_{\by \in \R^n_+} ~~ \frac{p(\by)}{\by_1\cdots\by_n} \nonumber \\ &=& \inf_{\by \in \R^n_+, ~ \by_1\cdots\by_n = 1} ~~ p(\by). \label{eq:capacity}
\end{eqnarray}
Gurvits then proves the following result:
\begin{theorem}\cite{Gur2008} \label{theo:gurvitscapacity}
For any matrix $\bM \in \R^{n \times n}_+$, we have that:
\begin{equation*}
\frac{n!}{n^n} \mathrm{cap}(p_\bM) \leq \mathrm{perm}(\bM) \leq \mathrm{cap}(p_\bM).
\end{equation*}
Here the polynomial $p_\bM$ and its capacity $\mathrm{cap}(p_\bM)$ are as defined in \eqref{eq:permpoly} and \eqref{eq:capacity}.  Further if each column of $\bM$ has at most $k$ nonzeros, then the factor in the lower bound can be improved from $\tfrac{n!}{n^n}$ to $\left(\tfrac{k-1}{k}\right)^{(k-1)n}$.
\end{theorem}
Gurvits in fact proves a more general statement involving so-called stable polynomials, but the above restricted version will suffice for our purposes.  The upper-bound in this statement is straightforward to prove; it is the lower bound that is the key technical novelty.  Thus, if one could compute the capacity of the polynomial $p_\bM$ associated to a nonnegative matrix $\bM$, then one can obtain an exponential-factor approximation of the permanent of $\bM$ as $\tfrac{n!}{n^n} \approx \exp(-n)$.  In order to compute the capacity of $p_\bM$ via GP, we apply the transformation $\bx = \log\{\by\}$ in \eqref{eq:capacity} and solve the following program\footnote{Such a logarithmic transformation of the variables is also employed in converting a GP specified in terms of non-convex posynomial functions to a GP in convex form; see \cite{BoyKVH2007,DufPZ1967} for more details.}:
\begin{eqnarray}
\mathrm{cap}(p_\bM) &=& \inf_{\bx \in \R^n} ~ \prod_{i=1}^n \left(\sum_{j=1}^n \bM_{i,j} \exp(\bx_j) \right) ~~ \mathrm{s.t.} ~~ \ones' \bx = 0 \nonumber \\ &=& \inf_{\bx \in \R^n, \bbeta \in \R^n, \delta \in \R} ~ \exp(\delta) ~~ \mathrm{s.t.} ~~ \ones' \bx = 0, ~ \sum_{i,j=1}^n \bM_{i,j} \exp(\bx_j - \bbeta_i - 1) + \mathbf{\ones}' \bbeta \leq \delta. \label{eq:capgp}
\end{eqnarray}
In obtaining the second formulation, we appeal to the expression \eqref{eq:prodsumexp3} for CGP representations of upper-bound constraints on a product of sums of exponentials.

\paragraph{Solution.} Focusing on a set $\cC \subset \R^{n \times n}_+$ of entrywise nonnegative matrices, our proposal to approximately maximize the permanent is to find $\bM \in \cC$ with maximum capacity $\mathrm{cap}(p_\bM)$, which leads to the following consequence of Theorem~\ref{theo:gurvitscapacity}:
\begin{proposition} \label{theo:permcapmax}
Let $\cC \in \R^{n \times n}_+$ be a set of nonnegative matrices, and consider the following two quantities:
\begin{eqnarray*}
\hat{\bM}_{\mathrm{perm}} &=& {\arg\sup}_{\bM \in \cC} ~ \mathrm{perm}(\bM) \\
\hat{\bM}_{\mathrm{cap}} &=& {\arg\sup}_{\bM \in \cC} ~ \mathrm{cap}(p_\bM)
\end{eqnarray*}
Then we have that
\begin{equation*}
\frac{n!}{n^n} \mathrm{perm}(\hat{\bM}_{\mathrm{perm}}) \leq \mathrm{perm}(\hat{\bM}_{\mathrm{cap}}) \leq \mathrm{perm}(\hat{\bM}_{\mathrm{perm}}).
\end{equation*}
The factor in the lower bound can be improved from $\tfrac{n!}{n^n}$ to $\left(\tfrac{k-1}{k}\right)^{(k-1)n}$ if every matrix in $\cC$ has at most $k$ nonzeros in each column.
\end{proposition}

\begin{proof}
The proof follows from a direct application of Theorem~\ref{theo:gurvitscapacity}:
\begin{equation*}
\frac{n!}{n^n} \mathrm{perm}(\hat{\bM}_{\mathrm{perm}}) \leq \frac{n!}{n^n} \mathrm{cap}(p_{\hat{\bM}_{\mathrm{perm}}}) \leq \frac{n!}{n^n} \mathrm{cap}(p_{\hat{\bM}_{\mathrm{cap}}}) \leq \mathrm{perm}(\hat{\bM}_{\mathrm{cap}}) \leq \mathrm{perm}(\hat{\bM}_{\mathrm{perm}}).
\end{equation*}
The first and third inequalities are a result of Theorem~\ref{theo:gurvitscapacity}, and the second and fourth inequalities follow from the definitions of $\hat{\bM}_{\mathrm{cap}}$ and of $\hat{\bM}_{\mathrm{perm}}$ respectively. The improvement in the lower bound also follows from Theorem~\ref{theo:gurvitscapacity}.
\end{proof}

In summary, maximizing the capacity with respect to the family $\cC$ provides an approximation to the permanent-maximizing element of $\cC$.  Specifically, for sets $\cC$ that have tractable conic representations, the next result provides a CGP solution for capacity maximization:

\begin{proposition} \label{theo:capmaxcgp}
Suppose that $\cC \subset \R^{n \times n}_+$ is defined as:
\begin{equation*}
\cC = \left\{\bM ~|~ \bM \geq \bzero, ~ \bg + \cL(\bM) \in \K \right\}
\end{equation*}
for $\cL: \R^{n \times n} \rightarrow \R^\ell$ a linear operator, for $\bg \in \R^\ell$, and for $\K \subset \R^\ell$ a convex cone.  Suppose further that $\cC$ is a compact subset, and that there exists a point that strictly satisfies the nonnegativity and conic constraints.  Then the problem of maximizing capacity with respect to the set $\cC$ can be solved via CGP as follows:
\begin{eqnarray*}
\sup_{\bM \in \cC} \mathrm{cap}(p_\bM) = \inf_{\bx \in \R^n, \bbeta \in \R^n, \bzeta \in \R^\ell} & \exp(\bg' \bzeta + \ones' \bbeta) & \\ \mbox{s.t.} & [\cL^\dag(\bzeta)]_{i,j} + \exp(\bx_j - \bbeta_i - 1) \leq 0, & \forall ~ i,j = 1,\dots,n \\ & \ones' \bx = 0 & \\ & \bzeta \in \K^\star &
\end{eqnarray*}
Here $\cL^\dag$ refers to the adjoint of the operator $\cL$.
\end{proposition}

\begin{proof}
By plugging in the expression \eqref{eq:capgp} for the capacity of a polynomial, we have:
\begin{eqnarray*}
\sup_{\bM \in \cC} \log(\mathrm{cap}(p_\bM)) &=& \sup_{\bM \in \cC} ~ \inf_{\bx, \bbeta \in \R^n, \ones' \bx = 0} ~ \sum_{i,j=1}^n \bM_{i,j} \exp(\bx_j - \bbeta_i - 1) + \ones' \bbeta \\ &=& \inf_{\bx, \bbeta \in \R^n, \ones' \bx = 0} ~ \sup_{\bM \in \cC} ~ \sum_{i,j=1}^n \bM_{i,j} \exp(\bx_j - \bbeta_i - 1) + \ones' \bbeta.
\end{eqnarray*}
In the second expression, the supremum and infimum are switched as the following conditions are satisfied \cite{Roc1970}: the objective function is concave in $\bM$ for fixed $\bx, \bbeta$ and convex in $\bx,\bbeta$ for fixed $\bM$, the constraints on $\bM$ and on $\bx,\bbeta$ are convex, and the constraint set $\cC$ for $\bM$ is compact.  Given the particular form of $\cC$ and appealing to Proposition~\ref{theo:robustgp}, we have the desired result.
\end{proof}

The steps of this proof follow for the most part from the discussion in Section~\ref{subsec:liftandproj} on constraints on products of functions, each of which is a maximum of a sum of exponentials (see also Section~\ref{subsec:maxsumexpcgp}).  The main novelty in the proof here is the switching of the infimum and supremum, which is based on saddle-point theory from convex analysis \cite{Roc1970}.


\subsection{Robust GP} \label{subsec:robustgp}
As our next application, we describe the utility of CGP in addressing the problem of computing solutions to GPs that are robust to variability in the input parameters of the GP.   Robust GPs arise in power control problems in communication systems as well as in robust digital circuit gate sizing, see e.g., \cite{BoyKPH2005}.

\begin{problem}
Given convex sets $\cC \subset \R^k_+$ and $\cQ \subset \R^{n \times k}$, obtain a tractable reformulation of the following constraint in a decision variable $\bx \in \R^n$:
\begin{equation}
\sup_{\bc \in \cC, \bQ \in \cQ} ~ \sum_{i=1}^k \bc_i \exp([\bQ' \bx]_i) \leq 1. \label{eq:rgp}
\end{equation}
\end{problem}

A robust GP is an optimization problem in which an affine function is minimized subject to affine constraints and constraints of the form \eqref{eq:rgp}.  Such optimization problems in which one seeks solutions that are robust to parameter uncertainty have been extensively investigated in the field of robust optimization \cite{BenEN2009,BenN1998}, and exact, tractable reformulations of robust convex programs are available in a number of settings, e.g., for robust LPs.  However, progress has been limited in the context of GPs for general convex uncertainty sets $\cC, \cQ$.

In their seminal work on robust convex optimization \cite{BenN1998}, Ben-Tal and Nemirovski obtained an exact, tractable reformulation of robust GPs in which a very restricted form of coefficient uncertainty is allowed in a sum-of-exponentials function -- specifically, the set $\cQ \subset \R^{n \times k}$ is a singleton set (no uncertainty) and $\cC \subset \R^k_+$ must be a convex ellipsoid in $\R^k_+$ specified by a quadratic form defined by an elementwise nonnegative matrix.  The reformulation given in \cite{BenN1998} for such robust GPs is itself a GP with additional variables, which can be solved efficiently.

In subsequent work, Hsiung et al \cite{HsiKB2008} considered sum-of-exponentials functions with the coefficients absorbed into the exponent as follows:
\begin{equation*}
\sup_{[\bd, \bQ] \in \cD} ~ \sum_{i=1}^k\exp([\bQ' \bx + \bd]_i) \leq 1
\end{equation*}
Therefore, the uncertainty in the coefficients and the exponents is ``coupled'' via the set $\cD \subset \R^{n \times (k+1)}$.  For such constraints with $\cD$ being either a polyhedral set or an ellipsoidal set, Hsiung et al \cite{HsiKB2008} obtain tractable but inexact reformulations via piecewise linear relaxations, with the reformulations again being GPs.

\paragraph{Solution.} In order to develop a framework that allows for flexible specifications of uncertainty sets in robust GP, yet one in which we can obtain \emph{exact, tractable} reformulations via CGP, we consider the following general setting.  Suppose the uncertainty set $\cC \in \R^k_+$ is \emph{any} convex set with a tractable conic representation, and suppose that $\cQ \in \R^{n \times k}$ is decoupled columnwise as follows:
\begin{equation}
\cQ = \{\bQ ~|~ \bQ^{(i)} \in \cQ^{(i)}\}, \label{eq:rgpdecouple}
\end{equation}
where $\bQ^{(i)} \in \R^n$ is the $i$'th column of $\bQ$ and $\cQ^{(i)} \subset \R^n$ is \emph{any} convex set with a tractable conic representation (this set specifies uncertainty in the $i$'th column of $\bQ$, independent of the other columns of $\bQ$).  For such forms of uncertainty, Proposition~\ref{theo:robustgp} provides a CGP reformulation of \eqref{eq:rgp}.  Specifically, when the uncertainty sets $\cC$ and $\cQ^{(i)}$ for $i=1,\dots,k$ have tractable conic representations, we can obtain an \emph{exact} reformulation of \eqref{eq:rgp} that is efficiently computed via CGP.

In contrast to previous results on robust GP, note that the form of uncertainty for which we obtain an efficient and exact CGP reformulation is significantly more general than that considered in \cite{BenN1998}, in which $\cC$ can only be a restricted kind of ellipsoidal set and $\cQ$ must be a singleton set (no uncertainty).  On the other hand, Hsiung et al \cite{HsiKB2008} consider robust GPs in a different form and with polyhedral or ellipsoidal uncertainty that may be coupled across $\cQ^{(i)}$, but their reformulation is inexact.  A crucial distinction of our approach relative to those in \cite{BenN1998} and in \cite{HsiKB2008} is that our reformulation is a CGP, while those described in \cite{BenN1998} and in \cite{HsiKB2008} are GPs.  Hence, the additional flexibility provided by CGPs via combined exponential/affine constraints and conic constraints -- and the appearance of additional ``lifting'' variables in both these types of constraints -- plays a crucial role in obtaining exact and tractable reformulations for fairly general types of uncertainty sets in robust GPs.

\paragraph{Example.} To illustrate these distinctions concretely, consider a specific instance of a robust GP constraint.  Suppose $\cC \subset \R^k$ is a convex set of the form:
\begin{equation*}
\cC = \{\bc ~|~ \bc \geq \bzero, ~ \bg + \bF \bc \succeq \bzero \}
\end{equation*}
for $\bg \in \Sym(\ell)$ and $\bF : \R^k \rightarrow \Sym(\ell)$ a linear operator.  Suppose also that each $\cQ^{(i)} \subset \R^n$ for $i = 1,\dots,k$ is a set of the form:
\begin{equation*}
\cQ^{(i)} = \left\{\bq^{(i)} ~|~ \bh^{(i)} + \bM^{(i)} \bq^{(i)} \succeq \bzero \right\}
\end{equation*}
for $\bh^{(i)} \in \Sym(m)$ and $\bM^{(i)} : \R^n \rightarrow \Sym(m)$ a linear operator.  With these uncertainty sets, we have from Proposition~\ref{theo:robustgp} that
\begin{eqnarray*}
&\sup_{\bc \in \cC, \bQ^{(i)} \in \cQ^{(i)}} ~ \sum_{i=1}^k \bc_i \exp([\bQ' \bx]_i) \leq 1& \\ &\Updownarrow& \\ & \exists \bzeta \in \Sym(\ell), ~ \btheta^{(i)} \in \Sym(m) ~\mathrm{for}~ i = 1,\dots,k ~ \mathrm{s.t.}& \\ &
\bg' \bzeta \leq 1, ~ \bzeta \succeq \bzero, & \\ &{\bM^{(i)}}^\dag \big(\btheta^{(i)} \big) + \bx = \bzero, ~ \exp\left({\bh^{(i)}}' \btheta^{(i)} \right) + \left[\bF^\dag \left(\bzeta \right)\right]_i \leq 0, ~ \btheta^{(i)} \succeq \bzero, ~\mathrm{for} ~ i = 1,\dots,k&
\end{eqnarray*}
Here ${\bM^{(i)}}^\dag$ and $\bF^\dag$ represent the adjoint of the operators ${\bM^{(i)}}$ and $\bF$ respectively.  Note that in this formulation, the uncertainty sets $\cC$ and $\cQ^{(i)}$ are essentially arbitrary SDP-representable sets.  Such forms of robust GP cannot be handled by previous approaches \cite{BenN1998,HsiKB2008}, but the expressive power of combined exponential/affine constraints and conic constraints enables us to obtain exact, tractable representations via CGP.

\subsection{Hitting Times of Dynamical Systems} \label{subsec:hittingtime}

Consider a linear dynamical system with state-space equations as follows:
\begin{equation}
\dot{\bx}(t)= \bA \bx(t), \label{eq:dynsystem}
\end{equation}
where the state $\bx(t) \in \mathbb{R}^n$ for $t \geq 0$.  We assume throughout this section that the transition matrix $\bA \in \R^{n \times n}$ is diagonal; otherwise, one can always change to the appropriate \emph{modal coordinates} given by the eigenvectors of $\bA$ (assuming $\bA$ is diagonalizable).  The diagonal entries of $\bA$ are called the \emph{modes} of the system.  Suppose that the parameters of the dynamical system can take on a range of possible values with $\bA \in \mathcal{A}$ and $\bx(0) \in \xinit$; the set $\mathcal{A}$ specifies a set of possible modes and the set $\xinit \subseteq \R^n$ specifies the set of possible initial conditions.  Suppose also that we are given a \emph{target set} $\xtarget \subseteq \mathbb{R}^n$, and we wish to find the smallest time required for the system to reach a state in $\xtarget$ from an arbitrary initial state in $\xinit$.  Formally, we define the \emph{worst-case hitting time} of the dynamical system \eqref{eq:dynsystem} to be:
\begin{equation}
\tau(\xinit, \xtarget, \mathcal{A}) \triangleq \inf \Big\{ t ~|~ t \geq 0 ~\mathrm{and}~ \bx(0) \exp \left\{\bA t \right\} \in \xtarget ~ \mathrm{for~all} ~ \bx(0) \in \xinit, \bA \in \mathcal{A} \Big\}. \label{eq:hittingtime}
\end{equation}
Indeed, for an initial state $\bx(0)$, the state of the system \eqref{eq:dynsystem} at time $t$ is given by $\bx(t)= \bx(0) \exp \left\{\bA t \right\}$; consequently, the quantity $\tau(\xinit, \xtarget,\mathcal{A})$ represents the amount of time that the worst-case trajectory of the system, taken over all initial conditions $\xinit$ and mode values $\mathcal{A}$, requires to enter the target set $\xtarget$.

\begin{problem}
Compute the hitting time $\tau(\xinit, \xtarget, \mathcal{A})$ \eqref{eq:hittingtime} of the linear dynamical system \eqref{eq:dynsystem} with the possible initial states, the target set, and the possible set of modes given by $\xinit$, $\xtarget$, and $\mathcal{A}$ respectively.
\end{problem}

Hitting times are of interest in system analysis and verification \cite{PraJ20004}.  As an example, suppose that a system has the property that $\tau(\xinit, \xtarget,\mathcal{A}) = \infty$; this provides a proof that from certain initial states in $\xinit$, the system never enters the target state $\xtarget$. On the other hand, if $\tau(\xinit, \xtarget,\mathcal{A}) = 0$, we have a certificate that $\xtarget \subseteq \xinit$.  While verification of linear systems has been extensively studied via Lyapunov and barrier certificate methods, the study of hitting times has received relatively little attention (with a few exceptions such as \cite{YazP2004}).  In particular, the approaches in \cite{YazP2004} can lead to loose bounds as the worst-case hitting time is computed based on box-constrained outer approximations of $\xinit$ and $\xtarget$.

\paragraph{Solution.} In this section we show that for a particular class of these problems, the hitting time can be computed \emph{exactly} via a CGP.  Specifically, we make the following assumptions regarding the structure of the dynamical system \eqref{eq:dynsystem}:

\begin{itemize}

\item The set of modes is given by
    \begin{equation*}
    \mathcal{A} = \left\{\bA ~|~ \bA ~ \mathrm{diagonal~with}~\bA_{j,j} \in \left[\ell_j, u_j \right] ~ \forall j=1,\dots,n \right\}
    \end{equation*}

\item The set of initial states $\xinit \subseteq \mathbb{R}^n_+$ is given by a convex set with a tractable representation via affine and conic constraints.  In particular, as in Proposition~\ref{theo:robustgp}, $\xinit$ is specified as follows:
    \begin{equation*}
    \xinit = \{\bx \in \R^n_+ ~|~ \bg + \bF \bx \in \K \}.
    \end{equation*}
    Here $\bg \in \R^m$, $\bF \in \R^{m \times n}$, and the cone $\K \subset \R^m$ is assumed to be tractable to compute.

\item $\xtarget \subset \mathbb{R}^n$ is representable as the intersection of finitely many half spaces:
    \begin{equation*}
    \xtarget = \{ \bx \in \R^n_+ ~|~ {\bc^{(i)}}' \bx \leq \bd_i, ~ i=1,\dots, k,~ \mathrm{where~each}~ \bc^{(i)} \in \R^n_+, \bd \in \R^k_+\}.
    \end{equation*}
\end{itemize}

\begin{figure}
\centering
\includegraphics[scale=0.3]{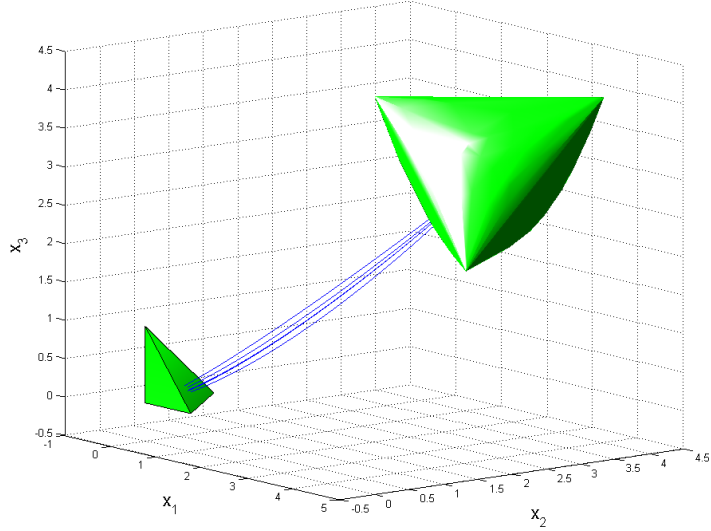}
\caption{Some sample trajectories of a linear system from $\xinit$ to $\xtarget$ for the dynamical system described in the example in Section~\ref{subsec:hittingtime}.  The (shifted) elliptope specifies the set of feasible starting points, while the tetrahedron specifies the target set.  The system consists of three modes.}
\label{fig:hitting_time}
\end{figure}

Under these conditions, the worst-case hitting time $\tau(\xinit, \xtarget,\mathcal{A})$ can be computed exactly via CGP.  In particular, note that $\tau(\xinit, \xtarget,\mathcal{A})$ can be written as:
\begin{eqnarray}
\tau(\xinit, \xtarget,\mathcal{A}) = \inf_{t \geq 0} ~ t ~~~ \mbox{s.t.} ~ \sup_{\bx(0) \in \xinit, \bA \in \mathcal{A}} ~ \sum_{j=1}^n \bc^{(i)}_j \bx_j(0) \exp\big(\bA_{j,j} t\big) \leq \bd_i, ~ i=1,\dots,k. \label{eq:hittingtimecgp}
\end{eqnarray}
Each constraint here can be reformulated as a collection of CGP-representable constraints via Proposition~\ref{theo:robustgp}.  Hence, for the particular problem structure we consider, the hitting time $\tau(\xinit, \allowbreak \xtarget, \allowbreak \mathcal{A})$ can be computed exactly via CGP.

\paragraph{Example.} Consider a dynamical system with three modes that each take on values in certain ranges as follows:
\begin{equation*}
\mathcal{A} = \left\{\bA \in \R^{3 \times 3} ~|~ \bA ~ \mathrm{diagonal}, ~ \bA_{1,1} \in [-0.4, -0.45], ~ \bA_{2,2} \in [-0.5, -0.6], ~ \bA_{3,3} \in [-0.6,-0.7] \right\}.
\end{equation*}
The set of initial states is a shifted \emph{elliptope} that is contained within the nonnegative orthant:
\begin{equation*}
\xinit =\left\{\bx(0) \in \R^3_+ ~:~ \left[ \begin{array}{ccc} 1 & \bx_1(0)-3 & \bx_2(0)-3 \\ \bx_1(0)-3 & 1 & \bx_3(0)-3 \\ \bx_2(0)-3 & \bx_3(0)-3 & 1 \end{array} \right] \succeq 0\right\}.
\end{equation*}
Finally, we let the target region be a tetrahedron:
\begin{equation*}
\xtarget = \left\{\bx \in \R^3_+ ~|~ \ones' \bx \leq 1 \right\}.
\end{equation*}
For these system attributes, a few sample trajectories are shown in Figure~\ref{fig:hitting_time}. We solve the CGP \eqref{eq:hittingtimecgp} and obtain that $\tau(\xinit,\xtarget,\mathcal{A})=7.6253$ is the worst-case hitting time.

\subsection{Quantum Channel Capacity} \label{subsec:quantumcapacity}

A quantum channel transmits quantum information, and every such channel has an associated \emph{quantum capacity} that characterizes the maximum rate at which quantum information can be reliably communicated across the channel.  We refer the reader to the literature for an overview of this topic \cite{Hol1997,NieC2011,SchW1998,Sho2003}.  For our purposes, we give a mathematical description of a quantum channel, and we describe methods for obtaining bounds on the associated quantum capacity via CGP.  The input and the output of a quantum channel are quantum states that are specified by positive semidefinite matrices with unit trace.  Such matrices are called \emph{quantum density matrices}.  Formally, a quantum channel is characterized by a linear operator that maps density matrices to density matrices; the dimensions of the input and output density matrices may, in general, be different:
\begin{equation*}
\cL : \Sym(n) \rightarrow \Sym(k)
\end{equation*}
For the linear operator $\cL$ to specify a valid quantum channel, it must map positive semidefinite matrices in $\Sym(n)$ to positive semidefinite matrices in $\Sym(k)$ and it must be trace-preserving.  These properties ensure that density matrices are mapped to density matrices.  In addition, $\cL$ must satisfy a property known as complete-positivity, which is required by the postulates of quantum mechanics -- see \cite{NieC2011} for more details.  Any linear map $\cL : \Sym(n) \rightarrow \Sym(k)$ that satisfies all these properties can be expressed via matrices\footnote{In full generality, density matrices are trace-one, positive semidefinite \emph{Hermitian} matrices, and $\bA^{(j)} \in \mathbb{C}^{k \times n}$.  As with SDPs, the CGP framework can handle linear matrix inequalities on Hermitian matrices, but we stick with the real case for simplicity.} $\bA^{(j)} \in \R^{k \times n}$ as follows \cite{HelK1970}:
\begin{equation*}
\cL({\boldsymbol \rho}) = \sum_j \bA^{(j)} {\boldsymbol \rho} {\bA^{(j)}}' ~~~ \mathrm{where} ~~~ \sum_j {\bA^{(j)}}' \bA^{(j)} = \mathbf{I}.
\end{equation*}
Letting $\Sp^n$ denote the $n$-sphere, the capacity of the channel specified by $\cL : \Sym(n) \rightarrow \Sym(k)$ is defined as:
\begin{equation}
C(\cL) \triangleq \sup_{\bv^{(i)} \in \Sp^n, ~ p_i \geq 0, ~ \sum_i p_i = 1} ~ H_{vn}\left[\cL\left(\sum_i p_i \bv^{(i)} {\bv^{(i)}}' \right)\right] - \sum_i p_i H_{vn}\left[ \cL \left(\bv^{(i)} {\bv^{(i)}}'\right) \right]. \label{eq:quantumcapacity}
\end{equation}
Here $H_{vn}$ is the Von-Neumann entropy \eqref{eq:vonneumann} and the number of states $\bv^{(i)}$ is part of the optimization.

\begin{problem}
Given a quantum channel specified by linear map $\cL$, compute the capacity of the channel.
\end{problem}

We note that there are several kinds of capacities associated with a quantum channel, depending on the protocol employed for encoding and decoding states transmitted across the channel; the version considered here is the one most commonly investigated, and we refer the reader to Shor's survey \cite{Sho2003} for details of the other types of capacities.  The quantity \eqref{eq:quantumcapacity} on which we focus here is called the $C_{1,\infty}$ quantum capacity -- roughly speaking, it is the capacity of a quantum channel in which the sender cannot couple inputs across multiple uses of the channel, while the receiver can jointly decode messages received over multiple uses of the channel.  Shor also describes in \cite{Sho2003} a procedure based on nonlinear programming to provide lower bounds on this quantum capacity.  As the first step of this method, one fixes a finite set of states $\left\{\bv^{(i)}\right\}_{i=1}^m$ with each $\bv^{(i)} \in \Sp^n$ and a density matrix ${\boldsymbol \rho} \in \Sym(n)$, so that ${\boldsymbol \rho}$ is in the convex hull of $\left\{\bv^{(i)} {\bv^{(i)}}' \right\}_{i=1}^m$. With these quantities fixed, one can obtain a lower bound on $C(\cL)$ by solving the following \emph{linear program}:
\begin{eqnarray}
C\left(\cL,\left\{\bv^{(i)}\right\}_{i=1}^m, {\boldsymbol \rho} \right) = \sup_{p_i \geq 0, ~\sum_{i=1}^m p_i = 1, ~{\boldsymbol \rho} = \sum_{i=1}^m p_i \bv^{(i)} {\bv^{(i)}}'} H_{vn}(\cL({\boldsymbol \rho})) - \sum_{i=1}^m p_i H_{vn}\left(\cL\left(\bv^{(i)} {\bv^{(i)}}' \right) \right). \label{eq:shorquantumapprox1}
\end{eqnarray}
In \cite{Sho2003}, Shor also suggests local heuristics to search for better sets of states and density matrices to improve the lower bound \eqref{eq:shorquantumapprox1}.  It is clear that $C\left(\cL, \left\{\bv^{(i)}\right\}_{i=1}^m, {\boldsymbol \rho} \right)$ as computed in \eqref{eq:shorquantumapprox1} is a lower bound on $C(\cL)$.  Indeed, one can check that:
\begin{equation}
C(\cL) = \sup_{\bv^{(i)} \in \Sp^n} ~ \sup_{{\boldsymbol \rho} \in \mathrm{conv}\left\{\bv^{(i)} {\bv^{(i)}}' \right\}} ~ C\left(\cL,\left\{\bv^{(i)}\right\}, {\boldsymbol \rho} \right). \label{eq:shorquantumapprox2}
\end{equation}
Here the number of states is not explicitly denoted, and it is also a part of the optimization.

\paragraph{Solution.} We describe an improvement upon the lower bound \eqref{eq:shorquantumapprox1} by proposing a solution based on a CGP.  Specifically, we observe that for a fixed set of states $\left\{\bv^{(i)}\right\}_{i=1}^m$ the following optimization problem is a CGP:
\begin{eqnarray}
C\left(\cL, \left\{\bv^{(i)}\right\}_{i=1}^m\right) &=& \sup_{{\boldsymbol \rho} \in \mathrm{conv}\left\{\bv^{(i)} {\bv^{(i)}}' \right\}_{i=1}^m} C\left(\cL, \left\{\bv^{(i)}\right\}_{i=1}^m, {\boldsymbol \rho}\right) \label{eq:cgpquantumapprox} \\ &=& \sup_{p_i \geq 0, ~ \sum_{i=1}^m p_i = 1} ~ H_{vn}\left[\cL\left(\sum_{i=1}^m p_i \bv^{(i)} {\bv^{(i)}}' \right)\right] - \sum_{i=1}^m p_i H_{vn}\left[ \cL \left(\bv^{(i)} {\bv^{(i)}}'\right) \right]. \nonumber
\end{eqnarray}
In particular, the optimization over density matrices ${\boldsymbol \rho} \in \mathrm{conv}\left\{\bv^{(i)} {\bv^{(i)}}' \right\}_{i=1}^m$ in \eqref{eq:shorquantumapprox2} can be folded into the computation of \eqref{eq:shorquantumapprox1} at the expense of solving a CGP instead of an LP.  It is easily seen from \eqref{eq:quantumcapacity}, \eqref{eq:shorquantumapprox2}, and \eqref{eq:cgpquantumapprox} that for a fixed set of states $\left\{\bv^{(i)}\right\}_{i=1}^m$:
\begin{equation*}
C(\cL) \geq C\left(\cL, \left\{\bv^{(i)}\right\}_{i=1}^m\right) \geq C\left(\cL, \left\{\bv^{(i)}\right\}_{i=1}^m, {\boldsymbol \rho}\right).
\end{equation*}
For a finite set of states $\left\{\bv^{(i)}\right\}_{i=1}^m$, the quantity $C\left(\cL, \left\{\bv^{(i)}\right\}_{i=1}^m\right)$ in \eqref{eq:cgpquantumapprox} is referred to as the \emph{classical-to-quantum capacity} with respect to the states $\left\{\bv^{(i)}\right\}_{i=1}^m$ \cite{Hol1997,SchW1998}; the reason for this name is that $C\left(\cL, \left\{\bv^{(i)}\right\}_{i=1}^m\right)$ is the capacity of a quantum channel in which the input is ``classical'' as it is restricted to be a convex combination of the finite set $\left\{\bv^{(i)} {\bv^{(i)}}' \right\}_{i=1}^m$.

One can improve upon the bound $C\left(\cL, \left\{\bv^{(i)}\right\}_{i=1}^m\right)$ by progressively adding more states to the collection $\left\{\bv^{(i)}\right\}_{i=1}^m$.  It is of interest to develop techniques to accomplish this in a principled manner.
\begin{figure}
\centering
\includegraphics[scale=0.2]{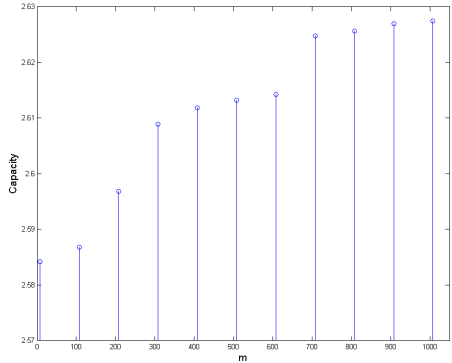}
\caption{A sequence of increasingly tighter lower bounds on the quantum capacity of the channel specified by \eqref{eq:quantumcapacityex}, obtained by computing classical-to-quantum capacities with increasingly larger collections of input states.}
\label{fig:capacity lower bound}
\end{figure}

\paragraph{Example.} We consider a quantum channel given by a linear operator $\cL: \Sym(8) \rightarrow \Sym(8)$ with:
\begin{equation}
\cL({\boldsymbol \rho}) = \bA^{(1)} {\boldsymbol \rho} {\bA^{(1)}}' + \epsilon^2 \bA^{(2)} {\boldsymbol \rho} {\bA^{(2)}}' + \bA^{(3)} {\boldsymbol \rho} {\bA^{(3)}}', \label{eq:quantumcapacityex}
\end{equation}
where $\bA^{(1)}$ is a random $8 \times 8$ diagonal matrix (entries chosen to be i.i.d. standard Gaussian), $\bA^{(2)}$ is a random $8 \times 8$ matrix (entries chosen to be i.i.d. standard Gaussian), $\epsilon \in [0,1]$, and $\bA^{(3)}$ is the symmetric square root of $\mathbf{I} - {\bA^{(1)}}' \bA^{(1)} - \epsilon^2 {\bA^{(2)}}' \bA^{(2)}$.  We scale $\bA^{(1)}$ and $\bA^{(2)}$ suitably so that $\mathbf{I} - {\bA^{(1)}}' \bA^{(1)} - \epsilon^2 {\bA^{(2)}}' \bA^{(2)}$ is positive definite for all $\epsilon \in [0,1]$.  For this quantum channel, we describe two numerical experiments.

In the first experiment, we compute the classical-to-quantum capacity of this channel with the input states $\bv^{(i)} \in \mathbb{S}^8$ for $i=1,\dots,8$ being the $8$ standard basis vectors in $\R^{8}$.  In other words, the input density ${\boldsymbol \rho} \in \Sym(8)$ is given by a diagonal matrix.  We add unit vectors in $\R^{8}$ at random to this collection, and we plot the increase in Figure~\ref{fig:capacity lower bound} in the corresponding classical-to-quantum channel capacity \eqref{eq:cgpquantumapprox} -- in each case, the capacity is evaluated at $\epsilon = 0.5$.  In this manner, CGPs can be used to provide successively tighter lower bounds for the capacity of a quantum channel.

Next, we compare the classical-to-quantum capacity of the channel \eqref{eq:quantumcapacityex} -- input states $\bv^{(i)} \in \mathbb{S}^8$ for $i=1,\dots,8$ being the $8$ standard basis vectors in $\R^{8}$ -- to the capacity of a purely classical channel induced by the quantum channel \eqref{eq:quantumcapacityex}.  Specifically, let $\mathfrak{P}_{\mathrm{diag}} : \Sym(8) \rightarrow \Sym(8)$ denote an operator that projects a matrix onto the space of diagonal matrices, i.e., zeros out the off-diagonal entries.  Then we consider the following classical channel induced by the quantum channel \eqref{eq:quantumcapacityex} for a diagonal density matrix ${\boldsymbol \rho}$:
\begin{equation}
\cL_{\mathrm{classical}}({\boldsymbol \rho}) \triangleq \mathfrak{P}_{\mathrm{diag}}(\cL({\boldsymbol \rho})).  \label{eq:inducedclassical}
\end{equation}
Both the inputs and outputs of this channel are diagonal density matrices, and they can therefore be interpreted as classical distributions.  Figure~\ref{fig:capacity comparison} shows two plots for $\epsilon$ ranging from $0$ to $1$ of both the classical-to-quantum capacity of \eqref{eq:quantumcapacityex} and the classical capacity of the induced classical channel \eqref{eq:inducedclassical}.  Note that for $\epsilon = 0$ the output of the operator \eqref{eq:quantumcapacityex} is diagonal if the input is given by a diagonal density matrix.  Therefore, the two curves coincide at $\epsilon = 0$ in Figure~\ref{fig:capacity comparison}.  For larger values of $\epsilon$, the classical-to-quantum channel is greater than the capacity of the induced classical channel.

\begin{figure}
\centering
\includegraphics[scale=0.4]{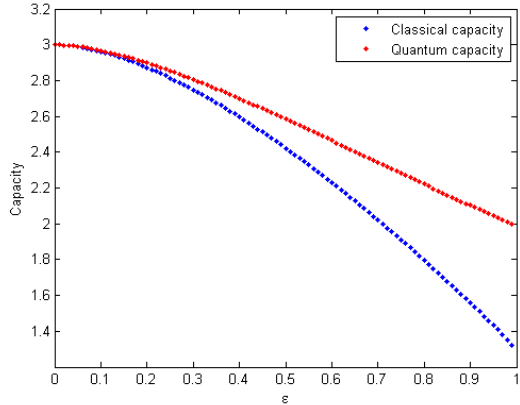}
\caption{Comparison of a classical-to-quantum capacity of the quantum channel specified by \eqref{eq:quantumcapacityex} and the classical capacity of a classical channel induced by the quantum channel given by \eqref{eq:quantumcapacityex}.}
\label{fig:capacity comparison}
\end{figure}

\subsection{Other Applications involving Von-Neumann Entropy Optimization} \label{subsec:vonneumann}

Several previously proposed methods in varied application domains can be viewed as special cases of CGP.  Specifically, the problem of maximizing Von-Neumann entropy of a positive semidefinite matrix subject to affine constraints on the matrix has been recognized as a useful conceptual and computational approach in a number of areas beyond the quantum communication setting of Section~\ref{subsec:quantumcapacity}.  We briefly discuss some of these applications here.  To be clear, although the proposed solutions in these domains can be viewed as particular instances of CGP, the framework developed in our paper is substantially more general and CGPs facilitate many new applications as described above.

\paragraph{Quantum state tomography} Quantum state tomography arises in the characterization of optical systems and in quantum computation \cite{GonLGCVM2013}.  The goal is to reconstruct a quantum state specified by a density matrix (symmetric positive semidefinite with unit trace) given partial information about the state.  Such information is typically provided via \emph{measurements} of the state, which can be viewed as linear functionals of the underlying density matrix.  As measurements can frequently be expensive to obtain, this is a type of inverse problem in which one must choose among the many density matrices that satisfy the limited measurement information that is acquired.  One proposal for tomography, based on the original work of Jaynes \cite{Jay1957} on the maximum entropy principle, is to find the Von-Neumann-entropy-maximizing density matrix among all density matrices consistent with the measurements \cite{GonLGCVM2013} -- the rationale behind this approach is that the entropy-maximizing matrix makes the least assumptions about the quantum state beyond the constraints imposed by the acquired measurements.  Using this method, the optimal reconstructed quantum state can be computed as the solution of a convex program that is a special instance of a CGP.

\paragraph{Equilibrium densities in statistical physics} In statistical mechanics, a basic objective is to investigate the properties of a system at equilibrium given information about macroscopic attributes of the system such as energy, number of particles, volume, and pressure.  In mathematical terms, the situation is quite similar to the previous setting with quantum state tomography -- specifically, the system is described by a density matrix (as in quantum state tomography, this matrix is symmetric positive semidefinite with unit trace), and the macroscopic properties can be characterized as constraints on linear functionals of this density matrix.  The Massieu-Planck extremum principle in statistical physics states that the system at equilibrium is given by the Von-Neumann-entropy-maximizing density matrix among all density matrices that satisfy the specified constraints \cite{ScoJ1977}.  As before, this equilibrium density can be computed efficiently via CGP.

\paragraph{Kernel learning} A commonly encountered problem in machine learning is to measure similarity or affinity between entities that may not belong to a Euclidean space, e.g., text documents.  A first step in such settings is to specify coordinates for the entities in a linear vector space after performing some nonlinear mapping, and then subsequently computing distances in the linear space.  Kernel methods approach this problem implicitly by working directly with inner-products between pairs of entities, thus combining the nonlinear mapping and the distance computation in a single step.  Therefore, given a finite collection of entities, a kernel is a symmetric positive definite matrix in which each entry specifies the inner-product between the corresponding pair of entities.  This approach leads to the natural question of learning a good kernel matrix.  Kulis et al \cite{KulSD2009} propose to minimize the following \emph{Von-Neumann relative entropy} between a kernel $\bM \in \Sym(n)$ (the decision variable) and a given reference kernel $\bM_0 \in \Sym(n)$ (specifying prior knowledge about the affinities among the $n$ entities) with $\bM_0 \succ \bzero$:
\begin{equation*}
D_{vn}(\bM, \bM_0) = - H_{vn}(\bM) - \mathrm{trace}\big(\bM \log(\bM_0)\big).
\end{equation*}
This minimization is subject to the decision variable $\bM$ being positive definite as well as constraints on $\bM$ that impose bounds on distances between pairs of entities in the linear space.  We refer the reader to \cite{KulSD2009} for more details, and for the virtues of minimizing the Von-Neumann relative entropy as a method for learning kernel matrices. In the context of this paper, we note that this problem can be viewed as a Von-Neumann entropy maximization problem subject to linear constraints, and it is therefore again a special instance of CGP.

In summary, we note that each of these problems can be viewed as a very special case of a CGP.  As such, the approach described in Section~\ref{subsec:cgpsolver} provides an effective method for solving these problems.  Although we do not discuss the details here, it is possible to improve upon the general techniques in Section~\ref{subsec:cgpsolver} by exploiting the particular structure of Von-Neumann entropy maximization problems subject to affine and semidefinite constraints.

\section{Discussion} \label{sec:discussion}

In this paper we introduced and investigated CGPs, a new class of structured convex optimization problems obtained by blending aspects of GPs and conic programs such as SDPs.  CGPs consist of conic constraints, affine constraints, and constraints on sums of exponential and affine functions.  The dual of a CGP can be expressed in a canonical manner as the joint maximization of the negative relative entropy between two nonnegative vectors, subject to affine and conic constraints.  CGPs can be solved efficiently in practice, and they expand significantly upon the expressive power of GPs and SDPs.  We demonstrate this point by describing a range of problems that are outside the scope of SDPs and GPs alone but for which CGPs provide effective solutions.  The application domains include permanent maximization, quantum capacity computation, hitting-time estimation in dynamical systems, and robust solutions of GPs.


There are several avenues for further research suggested by this paper.  It is of interest to develop efficient first-order methods to solve CGPs in order to scale to large problem instances.  Such massive size problems are especially prevalent in data analysis tasks, and are of interest in settings such as kernel learning.  On a related note, there exists a vast literature on exploiting the structure of a particular problem instance of an SDP or a GP -- e.g., sparsity in the problem parameters -- which can result in significant computational speedups in practice in the solution of these problems.  A similar set of techniques would be useful and relevant in all of the applications described in this paper.  Finally, we seek a deeper understanding of the geometry of CGP-representable convex sets such as their face structure, building on the extensive research in recent years on the geometry of SDP-representable convex sets \cite{BlePT2013,GouPT2013,HelN2009}.

\section*{Acknowledgements}
The authors would like to thank Pablo Parrilo for many enlightening conversations, and Leonard Schulman for pointers to the literature on Von-Neumann entropy.

\bibliography{cgprefs}

\begin{thebibliography}{10}

\bibitem{Ali1995}
F.~Alizadeh.
\newblock Interior-point methods in semidefinite programming with applications
  to combinatorial optimization.
\newblock {\em SIAM Journal on Optimization}, 5:13--51, 1995.

\bibitem{BapB1989}
R.~B. Bapat and M.~I. Beg.
\newblock Order statistics for nonidentically distributed variables and
  permanents.
\newblock {\em Sankhya: The Indian Journal of Statistics, A}, 51:79--93, 1989.

\bibitem{Bar1997}
A.~I. Barvinok.
\newblock Computing mixed discriminants, mixed volumes, and permanents.
\newblock {\em Discrete and Computational Geometry}, 18:205--237, 1997.

\bibitem{BenEN2009}
A.~Ben-Tal, L.~El~Ghaoui, and A.~Nemirovski.
\newblock {\em Robust optimization}.
\newblock Princeton University Press, 2009.

\bibitem{BenN1998}
A.~Ben-Tal and A.~Nemirovski.
\newblock Robust convex optimization.
\newblock {\em Mathematics of Operations Research}, 23:769--805, 1998.

\bibitem{BenN2001}
A.~Ben-Tal and A.~Nemirovskii.
\newblock {\em Lectures on Modern Convex Optimization}.
\newblock Society for Industrial and Applied Mathematics, 2001.

\bibitem{Bet1992}
U.~Betke.
\newblock Mixed volumes of polytopes.
\newblock {\em Archiv der Mathematik}, 58:388--391, 1992.

\bibitem{BlePT2013}
G.~Blekherman, P.~Parrilo, and R.~Thomas.
\newblock {\em Semidefinite Optimization and Convex Algebraic Geometry}.
\newblock Society for Industrial and Applied Mathematics, 2013.

\bibitem{BoyEFB1994}
S.~Boyd, L.~El~Ghaoui, E.~Feron, and V.~Balakrishnan.
\newblock {\em Linear Matrix Inequalities in System and Control Theory}.
\newblock Society for Industrial and Applied Mathematics, 1994.

\bibitem{BoyKPH2005}
S.~Boyd, S.~J. Kim, D.~Patil, and M.~Horowitz.
\newblock Digital circuit optimization via geometric programming.
\newblock {\em Operations Research}, 53:899--932, 2005.

\bibitem{BoyKVH2007}
S.~Boyd, S.~J. Kim, L.~Vandenberghe, and A.~Hassibi.
\newblock A tutorial on geometric programming.
\newblock {\em Optimization and Engineering}, 8:67--127, 2007.

\bibitem{BoyV2004}
S.~P. Boyd and L.~Vandenberghe.
\newblock {\em Convex {O}ptimization}.
\newblock Cambridge University Press, 2004.

\bibitem{Chi2005}
M.~Chiang.
\newblock Geometric programming for communication systems.
\newblock {\em Foundations and Trends in Communications and Information
  Theory}, 2:1--154, 2005.

\bibitem{ChiB2004}
M.~Chiang and S.~Boyd.
\newblock Geometric programming duals of channel capacity and rate distortion.
\newblock {\em IEEE Transactions on Information Theory}, 50:245--258, 2004.

\bibitem{DinKW1977}
J.~J. Dinkel, G.~A. Kochenberger, and S.~N. Wong.
\newblock Entropy maximization and geometric programming.
\newblock {\em Environment and Planning A}, 9:419--427, 1977.

\bibitem{DreJ1989}
J.~H. Drew and C.~R. Johnson.
\newblock The maximum permanent of a 3-by-3 positive semidefinite matrix, given
  the eigenvalues.
\newblock {\em Linear and Multilinear Algebra}, 25:243--251, 1989.

\bibitem{DufPZ1967}
R.~J. Duffin, E.~L. Peterson, and C.~M. Zener.
\newblock {\em Geometric Programming: Theory and Application}.
\newblock John Wiley and Sons, 1967.

\bibitem{Ego1981}
G.~P. Egorychev.
\newblock Proof of the {V}an der {W}aerden conjecture for permanents (english
  translation; original in russian).
\newblock {\em Siberian Mathematical Journal}, 22:854--859, 1981.

\bibitem{Fal1981}
D.~I. Falikman.
\newblock Proof of the {V}an der {W}aerden conjecture regarding the permanent
  of a doubly stochastic matrix (english translation; original in russian).
\newblock {\em Mathematical Notes}, 29:475--479, 1981.

\bibitem{Gli2000}
F.~Glineur.
\newblock An extended conic formulation for geometric optimization.
\newblock {\em Foundations of Computing and Decision Sciences}, 25:161--174,
  2000.

\bibitem{GoeW1995}
M.~Goemans and D.~Williamson.
\newblock Improved approximation algorithms for maximum cut and satisfiability
  problems using semidefinite programming.
\newblock {\em Journal of the ACM}, 42:1115--1145, 1995.

\bibitem{GonLGCVM2013}
D.~S. Gonçalves, C.~Lavor, M.~A. Gomes-Ruggiero, A.~T. Cesário, R.~O. Vianna,
  and T.~O. Maciel.
\newblock Quantum state tomography with incomplete data: Maximum entropy and
  variational quantum tomography.
\newblock {\em Physical Review A}, 87, 2013.

\bibitem{GouPT2013}
J.~Gouveia, P.~Parrilo, and R.~Thomas.
\newblock Lifts of convex sets and cone factorizations.
\newblock {\em Mathematics of Operations Research}, 38:248--264, 2013.

\bibitem{GraB2008}
M.~Grant and S.~Boyd.
\newblock Graph implementations for nonsmooth convex programs.
\newblock In V.~Blondel, S.~Boyd, and H.~Kimura, editors, {\em Recent Advances
  in Learning and Control}, Lecture Notes in Control and Information Sciences,
  pages 95--110. Springer-Verlag Limited, 2008.

\bibitem{GroJEW1986}
R.~Grone, C.~R. Johnson, S.~A. Eduardo, and H.~Wolkowicz.
\newblock A note on maximizing the permanent of a positive definite hermitian
  matrix, given the eigenvalues.
\newblock {\em Linear and Multilinear Algebra}, 19:389--393, 1986.

\bibitem{Gur2008}
L.~Gurvits.
\newblock Van der {W}aerden/{S}chrijver-{V}aliant like conjectures and stable
  (aka hyperbolic) homogeneous polynomials: one theorem for all.
\newblock {\em Electronic Journal of Combinatorics}, 15, 2008.

\bibitem{GurS2002}
L.~Gurvits and A.~Samorodnitsky.
\newblock A deterministic algorithm for approximating the mixed discriminant
  and mixed volume, and a combinatorial corollary.
\newblock {\em Discrete and Computational Geometry}, 27:531--550, 2002.

\bibitem{HelK1970}
K.~Hellwig and K.~Krauss.
\newblock Operations and measurements {II}.
\newblock {\em Communications of Mathematical Physics}, 16:142--147, 1970.

\bibitem{HelN2009}
J.~W. Helton and J.~Nie.
\newblock Sufficient and necessary conditions for semidefinite representability
  of convex hulls and sets.
\newblock {\em SIAM Journal on Optimization}, 20:759--791, 2009.

\bibitem{Hol1997}
A.~S. Holevo.
\newblock The capacity of the quantum channel with general signal states.
\newblock {\em IEEE Transactions on Information Theory}, 44:269--273, 1998.

\bibitem{HsiKB2008}
K.~L. Hsiung, S.~J. Kim, and S.~Boyd.
\newblock Tractable approximate robust geometric programming.
\newblock {\em Optimization and Engineering}, 9:95--118, 2008.

\bibitem{Jay1957}
E.~T. Jaynes.
\newblock Information theory and statistical mechanics.
\newblock {\em Physical Review Series II}, 106:620--630, 1957.

\bibitem{JerSV2004}
M.~Jerrum, A.~Sinclair, and E.~Vigoda.
\newblock A polynomial-time approximation algorithm for the permanent of a
  matrix with non-negative entries.
\newblock {\em Journal of the ACM}, 51:671--697, 2004.

\bibitem{KarMS1998}
D.~Karger, R.~Motwani, and M.~Sudan.
\newblock Approximate graph coloring by semidefinite programming.
\newblock {\em Journal of the ACM}, 45:246--265, 1998.

\bibitem{KulSD2009}
B.~Kulis, M.~Sustik, and I.~Dhillon.
\newblock Low-rank kernel learning with bregman matrix divergences.
\newblock {\em Journal of Machine Learning Research}, 10:341--376, 2009.

\bibitem{Lew1995}
A.~S. Lewis.
\newblock The convex analysis of unitarily invariant matrix functions.
\newblock {\em Journal of Convex Analysis}, 2:173--183, 1995.

\bibitem{LinSW2000}
N.~Linial, A.~Samorodnitsky, and A.~Wigderson.
\newblock A deterministic strongly polynomial algorithm for matrix scaling and
  approximate permanents.
\newblock {\em Combinatorica}, 20, 2000.

\bibitem{Min1984}
H.~Minc.
\newblock {\em Permanents}.
\newblock Cambridge University Press, 1984.

\bibitem{NesN1994}
Y.~Nesterov and A.~Nemirovski.
\newblock {\em Interior-Point Polynomial Algorithms in Convex Programming}.
\newblock Society of Industrial and Applied Mathematics, 1994.

\bibitem{NesT1998}
Y.~Nesterov and M.~Todd.
\newblock Primal-dual interior-point methods for self-scaled cones.
\newblock {\em SIAM Journal on Optimization}, 8:324--364, 1998.

\bibitem{NieC2011}
M.~Nielsen and I.~Chuang.
\newblock {\em Quantum Computation and Quantum Information}.
\newblock Cambridge University Press, 2011.

\bibitem{NocW1999}
J.~Nocedal and S.~Wright.
\newblock {\em Numerical Optimization}.
\newblock Springer, 1999.

\bibitem{PraJ20004}
S.~Prajna and A.~Jadbabaie.
\newblock Safety verification of hybrid systems using barrier certificates.
\newblock {\em Hybrid Systems: Computation and Control, Springer Lecture Notes
  in Computer Science}, 2004.

\bibitem{RecFP2010}
B.~Recht, M.~Fazel, and P.~A. Parrilo.
\newblock Guaranteed minimum rank solutions to linear matrix equations via
  nuclear norm minimization.
\newblock {\em SIAM Review}, 52:471--501, 2010.

\bibitem{Ren1997}
J.~Renegar.
\newblock {\em A Mathematical View of Interior-Point Methods in Convex
  Optimization}.
\newblock Society for Industrial and Applied Mathematics, 1997.

\bibitem{Roc1970}
R.~T. Rockafellar.
\newblock {\em Convex Analysis}.
\newblock Princeton University Press, 1970.

\bibitem{SchW1998}
B.~Schumacher and M.~D. Westmoreland.
\newblock Sending classical information via noisy quantum channels.
\newblock {\em Physical Review A}, 56:131--138, 1997.

\bibitem{ScoJ1977}
C.~H. Scott and T.~R. Jefferson.
\newblock Trace optimization problems and generalized geometric programming.
\newblock {\em Journal of Mathematical Analysis and Applications}, 58:373--377,
  1977.

\bibitem{Sha1982}
A.~Shapiro.
\newblock Weighted minimum trace factor analysis.
\newblock {\em Psychometrika}, 47:243--264, 1982.

\bibitem{Sho2003}
P.~W. Shor.
\newblock Capacities of quantum channels and how to find them.
\newblock {\em Mathematical Programming, B}, 97:311--335, 2003.

\bibitem{TutTT2003}
R.~H. Tutuncu, K.~C. Toh, and M.~J. Todd.
\newblock Solving semidefinite-quadratic-linear programs using {SDPT}3.
\newblock {\em Mathematical Programming, Series B}, 95:189--217, 2003.

\bibitem{VanB1994}
L.~Vandenberghe and S.~Boyd.
\newblock Semidefinite programming.
\newblock {\em SIAM Review}, 38:49--95, 1994.

\bibitem{VanBW1998}
L.~Vandenberghe, S.~Boyd, and S.~Wu.
\newblock Determinant maximization with linear matrix inequality constraints.
\newblock {\em SIAM Journal on Matrix Analysis and Applications}, 19:499--533,
  1998.

\bibitem{Vaz2004}
V.~Vazirani.
\newblock {\em Approximation Algorithms}.
\newblock Springer, 2004.

\bibitem{Von1937}
J.~Von~Neumann.
\newblock Some matrix-inequalities and metrization of matrix space.
\newblock {\em Tomsk University Review}, 1:286--300, 1937.

\bibitem{WalGW1986}
T.~Wall, D.~Greening, and R.~Woolsey.
\newblock Solving complex chemical equilibria using a geometric programming
  based technique.
\newblock {\em Operations Research}, 34:345--355, 1986.

\bibitem{YazP2004}
H.~Yazarel and G.~Pappas.
\newblock Geometric programming relaxations for linear system reachability.
\newblock In {\em Proceedings of the 2004 American Control Conference}, 2004.

\end{thebibliography}

\end{document}